\documentclass[11pt, reqno]{amsart}

\usepackage[utf8]{inputenc}
\usepackage[T2A]{fontenc}
\usepackage[ukrainian,english]{babel}
\usepackage{amsmath} % provides numberwithin (and lots more)
\usepackage{chngcntr}
\usepackage{mathtools}
\usepackage{amsthm}

\usepackage{thmtools}

\usepackage[english]{d-omega3}
\usepackage{amssymb}
\usepackage{graphicx}
\usepackage[all]{xy}

\def\div{{\rm div\,}}
\def\vecp{{\vec p}}
\DeclareMathOperator*{\esssup}{ess\,sup}

\usepackage{xurl}

\usepackage{csquotes} % serve a inserire le virgolette in una parte del titolo di un articolo in file.bib

\usepackage{bibentry}

\usepackage{listings}

\usepackage{enumitem}

\makeatletter
\def\namedlabel#1#2{\begingroup
    #2%
    \def\@currentlabel{#2}%
    \phantomsection\label{#1}\endgroup
}

\makeatother

\begin{document}

%\AtEndPreamble{%
%\theoremstyle{acmdefinition}
%\newtheorem{remark}[theorem]{Remark}}

%%%%%%%%%%%%%%%% Insert metadata of the paper

% \selectlanguage{ukrainian}
\selectlanguage{english}

% Set the title of the paper:
% If your title is long and will not fit to the top margin of the paper
% then specify a short title in square brackets

%\addtolength{\topmargin}{-8.54279pt}
%\setlength{\headheight}{22.54279pt}.

\title[\!Applications \!of \!Interpolation \!theory \!to \!the regularity of some
quasilinear PDEs]{Applications of Interpolation theory to the regularity of some
quasilinear PDEs}

% Your title may contain line breaks like '\\'
% In that case define the following command including there your title without such symbols '\\'
% \contentsTitle{}

%%% information on the first author:
%%  \author[Name which will be printed at the end of the paper]
%%         {Name printed in the paper title}

\author[I. Ahmed]{Irshaad Ahmed}
\organization{Department of Mathematics, Sukkur IBA University}
\address{Sukkur, PAKISTAN}
\email{irshaad.ahmed@iba-suk.edu.pk}
\orcid{0000-0003-0630-7620}

\author[A. Fiorenza]{Alberto Fiorenza}
\organization{University of Napoli \lq\lq Federico II\rq\rq\\Via Monteoliveto 3, 80134 Napoli, ITALY\\Istituto per le Applicazioni del Calcolo “Mauro Picone”, Consiglio Nazionale delle Ricerche\\Via Pietro Castellino 111, 80131 Napoli, ITALY}
%\address{Via Monteoliveto 3, 80134 Napoli, ITALY}
%\organization{Istituto per le Applicazioni del Calcolo “Mauro Picone”, Consiglio Nazionale delle Ricerche}
%\address{Via Pietro Castellino 111, 80131 Napoli, ITALY}
\email{fiorenza@unina.it}
\orcid{0000-0003-2240-5423}

\author[M. R. Formica]{Maria Rosaria Formica}
%\department{Department of Algebra}
\organization{\lq\lq Parthenope\rq\rq University of Napoli}
\address{Via Generale Parisi 13, Napoli, 80132, ITALY}
\email{mara.formica@uniparthenope.it}
\orcid{0000-0003-3962-9554}

\author[A. Gogatishvili]{Amiran Gogatishvili}
\organization{Institute of Mathematics of the Czech Academy of Sciences, Zitná}
\address{115 67 Prague 1, CZECH REPUBLIC}
\email{gogatish@math.cas.cz}
\orcid{0000-0003-3459-0355}

\author[A. El Hamidi]{Abdallah El Hamidi}
\organization{Département de Mathématiques et Laboratoire LaSIE, Université de La Rochelle}
\address{Av. Michel Crépeau 17042, La Rochelle Cedex 1, FRANCE}
\email{aelhamid@univ-lr.fr}

%\author[J. M. Rakotoson]{Jean Michel Rakotoson}

%%%%%%%%%%%% general information about the paper
% list of authors in the header of the paper
%\shortAuthorsList{M.R.~Formica}

%%%%%%%%%%%% similarly add information about all other authors

%%%%%%%%%%%% abstracts
\abstract{english}{We present some regularity results on the gradient of the weak or entropic-renormalized solution $u$ to the homogeneous Dirichlet problem for the quasilinear equations of the form
 \begin{equation*}\label{p-laplacian_eq}
 -\div(|\nabla u|^{p-2}\nabla u)+V(x;u)=f,
 \end{equation*}
where $\Omega$ is a bounded smooth domain of $\mathbb R^n$, $V$ is a nonlinear potential and  $f$ belongs to non-standard spaces like Lorentz-Zygmund spaces.

%The results, which have been exposed by the third author in a talk presented in AGMA 2024, the International Scientific Online Conference «Algebraic and geometric methods of analysis», May 27-30, 2024, Ukraine, constitute only a part of  results proved in detail in a paper coauthored with I. Ahmed, A. Fiorenza, A. Gogatishvili, A. El Hamidi and J.M. Rakotoson (\cite{AFFGER2023}).

Moreover, we collect some well-known and new results of the identication of some interpolation spaces and we enrich some contents with details.}

%\abstract{ukrainian}{Ми доводимо теорему Піфагора}

% \dedication{...}
\thanks{M. R. Formica is partially supported by the INdAM-GNAMPA project, {\itshape {Risultati di regolarit{à} per PDEs in spazi di funzione non-standard}}, CUP\_E53C22001930001 and partially supported by PRIN 2022 Project “Advanced theoretical aspects in PDEs and their applications”, code 2022BCFHN2.\\
\indent A. Gogatishvili is partially supported by the Czech Academy of Sciences (RVO 67985840), by the
Czech Science Foundation (GACR), grant no. 23-04720S, by the Shota Rustaveli
National Science Foundation (SRNSF), grant no. FR-21-12353, and by
the grant Ministry of Education and Science of the Republic of Kazakhstan
(project no. AP14869887).\\
\indent The authors wish to express their most sincere gratitude to Professor Jean Michel Rakotoson, who generously involved the whole team of authors in this field of research. He shared the brilliant ideas which have been the cornerstone of paper [2] and the major part of this project.} 

%https://www.overleaf.com/learn/latex/%5Chbadness
 
% 46M35 Abstract interpolation of topological vector spaces
% 35J62 Quasilinear elliptic equations
% 35B45 A priori estimates in context of PDEs
% 35D30 Weak solutions to PDEs
% 35J25 Boundary value problems for second-order elliptic equations
% 46E30 Spaces of measurable functions (Lp -spaces, Orlicz spaces, Köthe function spaces, Lorentz spaces, rearrangement invariant spaces, ideal spaces, etc.)
% 46B70 Interpolation between normed linear spaces

%
% \doi{}
% \received{}
% \accepted{}

\maketitle

%%% Text of the paper

\section{Introduction}

Let $\Omega$ be a bounded smooth domain of $\mathbb R^n$, $n\geq 2$. We study the regularity on the gradient of the solutions to the quasilinear Dirichlet problem
\begin{equation}\label{p-laplacian}
 -\Delta_p u+V(x;u)=f \ \ \ {\rm in} \ \Omega,  \ \ \ \ u=0 \ \ {\rm on} \ \ \partial \Omega,
 \end{equation}
where $\Delta_p u=\div(|\nabla u|^{p-2}\nabla u)$ is the $p$-Laplacian operator, $1<p<\infty$,
$V:\Omega\times \mathbb R \to\mathbb R$ is a Carath{é}odory function such that

\begin{enumerate}
\renewcommand{\labelenumi}{(\theenumi)}
\item[(H1)\label{H1}] the mapping $x\in\Omega\to V(x;\sigma)$ is in $L^\infty(\Omega)$  for every $\sigma\in\mathbb R$;
\item[(H2)\label{H2}] the mapping $\sigma\in\mathbb R\to V(x;\sigma)$ is continuous and
non decreasing for almost every $x\in \Omega$ and $V(x;0)=0$.
\end{enumerate}
The datum $f$ of the equation will be assumed in non-standard spaces, such as Lorentz-Zygmund spaces or $G\Gamma$ spaces.

Most of our regularity results are based on applications of results on nonlinear interpolation of $\alpha$-H\"olderian mappings between interpolation spaces with logarithm functors, which extend some results proved in \cite{Tartar1972BullFrance,Tartar1972} by Luc Tartar, who first gave interpolation results on nonlinear H\"olderian mappings (which include Lipschitz mappings) and applied them to PDEs. Other results concerning interpolation of Lipschitz operators and other applications of Interpolation theory, also in PDEs, were given in \cite{Cianchi-Mazja_European,Maligranda1984Studia,Maligranda1989,Maligranda_Persson_Wyller}.

 In Section \ref{Sec function spaces} we define the spaces involved and recall some properties. In Section \ref{Section_interpolation_spaces} we define the interpolation spaces with logarithm function and we identify some interpolation spaces between couples of Lebesgue or Lorentz spaces,
recovering spaces as Lorentz–Zygmund spaces or $G\Gamma$-spaces.
In Section \ref{Section_Holderian_mappings} we give results concerning interpolation of H\"olderian mappings to couple of spaces with a logarithm function and in Section \ref{SecKfunctional} we study the action of these mappings on those couples. Finally in Section \ref{SecPDEs}, \ref{SecApplHolderianpgreater2} and \ref{SecApplHolderianpsmaller2} we provide applications of these results to regularity on the gradient of the weak or entropic-renormalized solution $u$ to quasilinear Dirichlet problem \eqref{p-laplacian}.
 %\begin{equation*}
% -\div(\widehat a(\nabla u ))+V(x;u)=f,  \ \ \ \ u=0 \ {\rm on} \ \partial \Omega,
% \end{equation*}
% associated to the Dirichlet homogeneous condition on the boundary, where $\Omega$ is a bounded smooth domain of $\mathbb R^n$,  $\widehat a(\nabla u)=|\nabla u|^{p-2}\nabla u$, $V$ is a nonlinear potential and  $f$ belongs to non-standard spaces like Lorentz-Zygmund spaces.
We also show that the mapping ${\mathcal T}: \ {\mathcal T}f=\nabla u$ is locally or globally $\alpha$-H\"olderian under suitable values of $\alpha$ and appropriate assumptions on $V$. We will exibit only some proofs, to give an idea of the tools involved.
%
% Furthermore, also the anisotropic version or the variable exponents version of the Laplacian are considered.

\section{Definitions of some functional spaces}\label{Sec function spaces}

Let $\Omega$ be a bounded open set of $\mathbb R^n$, with Lebesgue measure $|\Omega|$.

Here and in the sequel, if $A_1$ and $A_2$ are two quantities depending on some parameters, we
write $A_1\lesssim  A_2$ if there exists $c>0$ independent of the parameters such that
$A_1\leq cA_2$, and $A_1\simeq A_2$ if and only if $A_1\lesssim A_2$ and $A_2\lesssim A_1$.

\vspace{1mm}

Moreover sometimes, for simplicity of notations, for a function space $X$ on $\Omega$ we will only write $X$ instead of $X(\Omega)$.

\begin{definition} \textbf{Decreasing rearrangement.}\label{decreasing_rearrangement}\\
For a measurable function $f:\Omega\to \mathbb R$, for any $t\geq 0$, the distribution function of $f$ is
\begin{equation*}\label{distribuction_function}
D_f(t)=|\{x\in \Omega \, : \, |f(x)|\geq t\}|,
\end{equation*}
and the decreasing rearrangement of $f$ is defined by
\begin{equation*}\label{rearrangement}
f_*(s)=\inf\{t \, : \, D_f(t)\leq s\}, \ \ \forall\, s\in (0,|\Omega|),
\end{equation*}
(with the convention $\inf \emptyset=+\infty$).

\vspace{1mm}
The maximal function $f_{**}$ of $f$ is defined by
\begin{equation*}\label{maximal_function}
f_{**}(s)=\frac{1}{s}\int_0^s f_*(t)\,dt.
\end{equation*}

\end{definition}

\vspace{2mm}

\begin{definition} \textbf{Lorentz spaces} (see, e.g., \cite{Bennett_Sharpley}). \label{def_Lorentz_space}\\
Let $1\leq p<+\infty$ and $1\leq q\leq +\infty$. The Lorentz space $L^{p,q}(\Omega)$ is defined as the set
of all measurable functions $f$ on $\Omega$ for which the quantity
\begin{equation*}\label{Lorentz quasi norm}
||f||_{p,q}=||f||_{L^{p,q}(\Omega)}=||t^{\frac{1}{p}-\frac{1}{q}}f_*(t)||_{L^q(0,|\Omega|)}
\end{equation*}
is finite, where $||\cdot||_{L^q(0,|\Omega|)}$ is the norm in the Lebesgue space $L^q$.

For $1< p<+\infty$, the following equivalence holds,
\begin{equation*}\label{Lorentz norm}
||f||_{p,q}\simeq ||t^{\frac{1}{p}-\frac{1}{q}}f_{**}(t)||_{L^q(0,|\Omega|)}=\left\{
             \hspace{-0.2cm} \begin{array}{ll}
                \displaystyle \left(\int_0^{|\Omega|}[t^{\frac{1}{p}}f_{**}(t)]^q\, \frac{dt}{t} \right)^{\frac{1}{q}} &  \hspace{-1mm}\hbox{if} \  1\leq q<+\infty \\
                \displaystyle \sup_{0<t<|\Omega|}t^{\frac{1}{p}}f_{**}(t)  &  \hspace{-1mm} \hbox{if} \  q=+\infty
              \end{array}
            \right.
\end{equation*}
up to multiplicative constants.
\end{definition}
In particular,
$$L^{p,p}(\Omega)=L^p(\Omega).$$
The inclusions among Lorentz spaces are given by
%$$L^p(\Omega)\subset L^{p,q}(\Omega), \ \ \ \hbox{if} \ p<q; \ \ \ L^{p,q}(\Omega)\subset L^p(\Omega) , \ \ \ \hbox{if} \ p>q ;$$
$$L^{p,q}(\Omega)\subset L^{p,r}(\Omega), \ \ \ \hbox{if} \ 1<p<+\infty, \ \ 1\leq q<r\leq +\infty;
 $$
$$ L^{p,q}(\Omega)\subset L^{r,s}(\Omega), \ \ \ \hbox{if} \  1<r<p<+\infty, \ \ 1\leq q,s\leq +\infty.  $$
In particular, for $1\leq q<p<r\leq +\infty$, we have
$$L^r(\Omega)\subset L^{p,q}(\Omega)\subset
 L^{p}(\Omega)\subset L^{p,r}(\Omega)\subset
L^{p,\infty}(\Omega)\subset L^{q}(\Omega).$$

\vspace{2mm}

\begin{definition}\textbf{Lorentz-Zygmund spaces} (see, e.g., \cite{Bennett_Sharpley}). \label{def_Lorentz-Zygmund_space}\\
Assume now, for simplicity, $|\Omega|=1$, $0< p,q\leq +\infty, \ \lambda\in\mathbb R$. The Lorentz-Zygmund space $L^{p,q}(\log L)^\lambda(\Omega)$ consists of all Lebesgue measurable functions $f$ on $\Omega$ such that
\begin{equation*}
\|f\|_{L^{p,q}(\log L)^\lambda(\Omega)}=
\left\{
     \begin{array}{ll}
     \left(\displaystyle\int_0^1\left[t^{\frac
1p}(1-\log t)^\lambda f_*(t)\right]^q
\frac{dt}{t}\right)^{\frac1q} , & \hbox{if} \
0< q<+\infty\\
     \displaystyle\sup_{0<t<1}t^{\frac{1}{p}}(1-\log t)^\lambda f_*(t)\, ,  \ & \hbox{if} \
\ q=+\infty
   \end{array}
 \right.
\end{equation*}
is finite.
\end{definition}
These spaces reduce to the Lorentz spaces for $\lambda=0$ and to the Zygmund spaces for $p=q$, \ $0<p<+\infty$:
$$ L^{p,q}(\log L)^0(\Omega)=L^{p,q}(\Omega), \ \ \ 0<p,q\leq +\infty, \ \ \ (\lambda=0)$$
$$ L^{p,p}(\log L)^\lambda(\Omega)=L^{p}(\log L)^\lambda(\Omega), \ \ \ 0<p<+\infty, \ \ \ (\lambda\in \mathbb R).$$

\vspace{1mm}
When $p=1$ and $\lambda=1$ the Zygmund space $L^1(\log L)^1(\Omega)$ is also denoted simply by $L(\log L)(\Omega)$.

\vspace{2mm}

\noindent \textbf{Inclusion relations} (see \cite[pp.31-33]{Bennett_Rudnick}): for $0<p\leq+\infty$, \ $0<q,s\leq +\infty$, \ $\lambda_1,\lambda_2\in \mathbb R$,

\vspace{0.1cm}

\begin{itemize}
\item  $ \displaystyle L^{p,q}(\log L)^{\lambda_1}\subset L^{r,s}(\log L)^{\lambda_2}, \ \ \ 0<r<p\leq+\infty;$
\vspace{2mm}

\item  $\displaystyle L^{p,q}(\log L)^{\lambda_1}\subset L^{p,s}(\log L)^{\lambda_2} $

\noindent whenever either $q\leq s, \ \lambda_1\geq \lambda_2$ \ or \  $\ q>s, \  \lambda_1+\dfrac{1}{q}>\lambda_2+\dfrac{1}{s}$;

\item $\displaystyle L^{\infty,q}(\log L)^{\lambda_1}\subset L^{\infty,s}(\log L)^{\lambda_2}, \ \ 0<q<s\leq \infty, \ \ \lambda_1+\dfrac{1}{q}=\lambda_2+\dfrac{1}{s}$.

\end{itemize}
In particular, for $0<q<p<r<\infty$ and for any $\varepsilon>0$, as consequence of the previous results, the following chain of inclusions holds (see \cite[p.192]{Persson1983}:
\begin{equation*}
L^p(\log L)^{\frac{1}{q}-\frac{1}{p}+\varepsilon}\subset L^{p,q}\subset L^p \subset L^{p,q}(\log L)^{\frac{1}{p}-\frac{1}{q}-\varepsilon},
\end{equation*}
\begin{equation*}
L^{p,r}(\log L)^{\frac{1}{p}-\frac{1}{r}+\varepsilon}\subset L^p\subset L^{p,r} \subset L^p(\log L)^{\frac{1}{r}-\frac{1}{p}-\varepsilon}.
\end{equation*}
All the previous inclusions are the sharpest possible in the sense that we can nowhere permit
that $\varepsilon=0$.

\vspace{2mm}

\begin{definition} \textbf{Grand and small Lebesgue spaces} (see \cite{Greco-Iwaniec-Sbordone,DiFratta_Fiorenza_2009,FFGKR_NA2018,Fiorenza-Formica-Gogatishvili-DEA2018} and references therein).\label{def_small_spaces}\\
Suppose $|\Omega|=1$. Let $1<p<\infty$ and $\alpha>0$. The grand Lebesgue space $L^{p),\alpha}(\Omega)$ is the set of all measurable functions such that the norm
\begin{equation*}
||f||_{L^{p),\alpha}(\Omega)}=||f||_{p),\alpha}=\sup_{0<\varepsilon<p-1}\left(\varepsilon^\alpha\int_\Omega |f|^{p-\varepsilon}\,dx\right)^{\frac{1}{p-\varepsilon}}
\end{equation*}
is finite. For $\alpha=1$ this space was defined by C. Sbordone  and T. Iwaniec in \cite{Iwaniec-Sbordone_Integrability} and it is denoted by $L^{p),1}=L^{p)}$.

The small Lebesgue space $L^{(p',\alpha}(\Omega)$, where $p'=\frac{p}{p-1}$, was introduced by A. Fiorenza in \cite{Fiorenza2000Collectanea} as the associate space to the grand Lebesgue space $L^{p),\alpha}(\Omega)$, that is $||\cdot||_{L^{(p',\alpha}(\Omega)}=||\cdot||_{(p',\alpha}$ is the
smallest functional defined on the measurable functions such that a kind of H\"older type inequality holds:
$$\int_\Omega f(x)\,g(x)\,dx\leq ||f||_{L^{p),\alpha}(\Omega)}||g||_{L^{(p',\alpha}(\Omega)}.$$
We have therefore $(L^{p),\alpha})^{'}=L^{(p',\alpha}$.
%defined through the function norm
%$$ ||g||_{L^{(p',\alpha}(\Omega)}=||g||_{(p',\alpha}=\sup\left\{\int_\Omega f\,g\,dx \, : \, f:\Omega\to \mathbb R \ \hbox{measurable}, ||f||_{p),\alpha}\leq 1\right\}.$$
The equivalence with the following quasi-norms hold (see \cite{Fiorenza-Karadzhov,DiFratta_Fiorenza_2009}):
\begin{eqnarray*}
||f||_{L^{p),\alpha}(\Omega)} &\simeq &\sup_{0<t<1}(1-\log t)^{-\frac{\alpha}{p}} \left(\int_t^1 f^p_*(\sigma)\,d\sigma\right)^{\frac{1}{p}}\\
||g||_{L^{(p',\alpha}(\Omega)}&\simeq& \int_0^1 (1-\log t)^{\frac{\alpha}{p}-1}\left(\int_0^t f^{p'}_*(\sigma)\, d\sigma\right)^{\frac{1}{p'}}\,\frac{dt}{t}.
\end{eqnarray*}
\end{definition}
The space $L^{(p,1}$ is denoted by $L^{(p}$.

These spaces and their generalizations have an important role in applications to PDEs, in Calculus of Variations, and also in Probability and Statistics (see, e. g, \cite{ForKozOstr_Lithuanian,FOS_MathNach,FOS2021_JMAA} and references therein).

\vspace{2mm}

\begin{definition} \textbf{Generalized Gamma spaces with double weights} (see \cite{FFGKR_NA2018,AFFGR2020}). \label{Gamma space}\\
Suppose $|\Omega|=1$. Let $w_1,\ w_2$ be two weights on $(0,1)$, $m\in[1,+\infty]$, $1\leq p<+\infty$. Assume the conditions:

 \vspace{3mm}

\noindent \textbf{(c1)} \ $\exists \,k>0 \, : \, w_2(2t)\leq Kw_2(t), \ \forall t\in(0,1/2)$ ;\\[1mm]
\noindent \textbf{(c2)} \ $ \displaystyle \int_0^t w_2(\sigma)d\sigma \in L^{\frac mp}(0,1;w_1)$.
\vspace{1mm}

The generalized Gamma space $G\Gamma(p,m;w_1,w_2)(\Omega)=G\Gamma(p,m;w_1,w_2)$ with double weights is the set of all measurable functions $f:\Omega \to \mathbb R$ such that
$$\left(\int_0^t f_*^p(\sigma)w_2(\sigma)\,d\sigma\right)^{\frac{1}{p}}\in L^{m}(0,1;w_1),$$
%$$G\Gamma(p,m;w_1,w_2)=\left\{f:\Omega \to \mathbb R \ \hbox{measurable} \, : \, \int_0^t f_*^p(\sigma)w_2(\sigma)\,d\sigma\in L^{\frac{m}{p}}(0,1;w_1)\right\}$$
endowed with the quasi-norm
$$ ||f||_{G\Gamma(p,m;w_1,w_2)}=\left[ \int_0^1 w_1(t)\Big(\int_0^t f_*^p(\sigma)w_2(\sigma)d\sigma\Big)^{\frac m p}dt \right]^{\frac 1 m},$$
with the usual change when $m=\infty$.

 \vspace{3mm}
If $w_2=1$ we denote $G\Gamma(p,m;w_1,1)=G\Gamma(p,m;w_1).$

\end{definition}

In \cite{AFF2022_JFAA} we define the $G\Gamma(p,m;w_1,w_2)$ spaces with a slight modification.

The scale of $G\Gamma(p,m;w_1,w_2)$ spaces is very general and covers many well-known
scales of spaces. In next examples we collect some particular cases.

\vspace{2mm}

\textbf{Examples.}

\begin{enumerate}

\item If $w$ is an integrable weight on $(0,1)$, then $G\Gamma(p,m;w)=L^p$ (see \cite[Remark 1, p.798]{Fiorenza_Rakotoson2008}).

\vspace{3mm}

\item Let $1\leq p,q <\infty$, \ $w_1$ an integrable weight on $(0,1)$, $w_2(t)=t^{\frac{q}{p}-1}$. Then
$$G\Gamma(q,m;w_1,w_2)=L^{p,q}.$$
In fact, the conditions \textbf{(c1)} and \textbf{(c2)} are satisfied.  The first is obvious, the second derives from
$$
\int_0^1 w_1(t) \left(\int_0^t \sigma^{\frac{q}{p}-1}\,d\sigma\right)^{\frac{m}{q}}dt\simeq \int_0^1 t^{\frac{m}{p}}w_1(t)\,dt \leq \int_0^1 w_1(t)\, dt<\infty.
$$

Now, if $f\in L^{p,q}$, we have
\begin{eqnarray*}
||f||_{G\Gamma(q,m;w_1,w_2)} &=& \left[ \int_0^1 w_1(t)\Big(\int_0^t f_*^q(\sigma)\, \sigma^{\frac{q}{p}-1}\,d\sigma\Big)^{\frac m q}dt \right]^{\frac 1 m}\\
& \leq & ||f||_{L^{p,q}}\left(\int_0^1 w_1(t)\,dt\right)^{\frac{1}{m}}<\infty,
\end{eqnarray*}
since $w_1$ is integrable on $(0,1)$. Therefore $f\in G\Gamma(q,m;w_1,w_2)$.

\vspace{1mm}

\noindent For the reverse inclusion $G\Gamma(q,m;w_1,w_2)\subseteq L^{p,q}$  see \cite[Prop. 2.1]{AFFGR2020}.

\vspace{3mm}

\item If $m=p$, \ $w(t)=t^{-1}(1-\log t)^{\theta p-2}$ and $\theta>\frac{1}{p}$, we have
$$G\Gamma(p,p;t^{-1}(1-\log t)^{\theta p-2})=L^p(\log L)^{\theta-\frac{1}{p}}$$
(see \cite[Proposition 6.2]{FFGKR_NA2018}),
which can be written also in this form
$$G\Gamma(p,p;t^{-1}(1-\log t)^{\theta p-1})=L^p(\log L)^{\theta}$$
(see \cite[Lemma 3.5]{Dominguez}).
% aggiungo 1 all'esponente del peso e poi divido per p

\vspace{3mm}

\item $G\Gamma(p,1;t^{-1}(1-\log t)^{\frac{\theta}{p'}-1})=L^{(p,\theta}$, \ \ $p'=\frac{p}{p-1}$ \ (see \cite[p.813]{FFR-DIE2017}).

\vspace{3mm}

\item If $m=1,\  p>1, \ \gamma, \beta\in \mathbb R$, $\gamma>-1$, \ $\gamma +\frac{\beta}{p}+1>0$,  \\ $\theta=p'\left(\gamma +\frac{\beta}{p}+1\right)$, \ $w_1(t)=t^{-1}(1-\log t)^{\gamma}$, $w_2(t)=(1-\log t)^\beta$, then
$$G\Gamma(p,1; w_1,w_2)=L^{(p,\theta}$$
(see \cite[Corollary 2.7, p.10]{AFFGR2020}).

\vspace{3mm}

\item If  $m,p\in[1,\infty)$, \ $\gamma, \beta\in \mathbb R$, $\gamma>-1$, \ $\gamma +\beta\frac{ m}{p}+1<0$,  \\ \ $w_1(t)=t^{-1}(1-\log t)^{\gamma}$, $w_2(t)=(1-\log t)^\beta$, then
$$||f||^m_{G\Gamma(p,m; w_1,w_2)}\simeq \int_0^1 (1-\log t)^{\gamma +\beta\frac{ m}{p}+1}\left(\int_t^1 f_*^p(x)\,dx\right)^{\frac{m}{p}}\,\frac{dt}{t} $$
(see \cite[Lemma 2.4, p.9]{AFFGR2020}).
\end{enumerate}

Here we collect some properties of inclusion among the spaces defined above.
%\begin{equation*}
%\bigcup_{0<\varepsilon<p-1}L^{p+\varepsilon}\subsetneq L^{(p,\theta}\subsetneq
%L^p\subsetneq L^{p),\theta}\subsetneq
% \bigcap_{0<\varepsilon<p-1}L^{p-\varepsilon}
%\end{equation*}
%\vspace{0.2cm}
\begin{equation*}
L^{(p}\subsetneq L^p\subsetneq
L^{p,\infty}\subsetneq
L^{p)}\subsetneq L^{p,\infty}(\log L)^{-\frac{1}{p}} \ \ \hbox{(see \cite[p.669]{Fiorenza-Karadzhov})}.
\end{equation*}
\begin{equation*}
\bigcup_{\varepsilon>0} L^{p+\varepsilon} \subsetneq
  \bigcup_{\beta>1} L^p (\log L)^{\frac{\beta\theta}{p'-1}}
\subsetneq
L^{(p,\theta}(\Omega)\subsetneq
L^p (\log L)^{\frac{\theta}{p'-1}}\subsetneq L^{p}(\Omega)
\end{equation*}
(see \cite[p.26]{Fiorenza-Formica-Gogatishvili-DEA2018}).
\begin{equation*}
L^p\subsetneq L^p(\log L)^{-\theta}\subsetneq
L^{p),\theta}\subsetneq
\bigcap_{\alpha>1} L^p(\log L)^{-\alpha\theta}\subsetneq
\bigcap_{0<\varepsilon<p-1}L^{p-\varepsilon}
\end{equation*}
(see \cite[p.24]{Fiorenza-Formica-Gogatishvili-DEA2018}).

\vspace{2mm}

\section{Interpolation spaces with logarithm function}\label{Section_interpolation_spaces}

\subsection{Preliminaries}
Let $(X_0,||\cdot||_0),\ (X_1, ||\cdot||_1)$ be  two normed spaces continuously embedded in a same Hausdorff topological vector space, that is, $(X_0,X_1)$ is  a compatible couple.

For $f\in X_0+X_1,\ t>0$, the \emph{Peetre $K$-functional} is defined by
$$K(f,t)\dot=K(f,t;X_0,X_1)=\inf_{f=f_0+f_1}\Big(||f_0||_{X_0}+t||f_1||_{X_1}\Big),$$
where the infimum extends over all decompositions $f=f_0+f_1$ of $f$ with $f_0\in X_0$ and $f_1\in X_1$.

\vspace{1mm}

For $0\leq\theta\leq 1,\ \ 1\leq q\leq+\infty,\ \ \alpha\in\mathbb R,$ the \emph{logarithmic interpolation space} $(X_0,X_1)_{\theta,q;\alpha}$ is the set of all functions $f\in X_0+X_1$ such that
$$ ||f||_{\theta,q;\alpha}=||t^{-\theta-\frac1q}(1-\log t)^\alpha K(f,t)||_{L^q(0,1)}<\infty.$$
For the limiting cases it only makes sense to consider  $\theta=0$ and $\alpha\geq -1/q$ or $\theta=1$ and $\alpha<-1/q$ or $\theta=1, q=\infty$ and $\alpha=0$.

% vedi Leo R. Ya. Doktorski, AN APPLICATION OF LIMITING INTERPOLATION TO THE FOURIER SERIES THEORY p.4

\vspace{1mm}

We refer, for example, to \cite{Bergh_Lofstrom,Evans_Opic_Pick,Gogatishvili_Opic_Trebels}.

\vspace{1mm}

The spaces $(X_0,X_1)_{0,q;\alpha}$ and $(X_0,X_1)_{1,q;\alpha}$ produce spaces which are “very close” to $X_0$ and to $X_1$ respectively.

\vspace{1mm}

If $X_1\subset X_0$ and $\alpha=0$, the space $(X_0,X_1)_{\theta,q;0}$ reduces to the classical real interpolation space $(X_0,X_1)_{\theta,q}$ defined by J. Peetre.

\vspace{1mm}

We recall some properties (see \cite[p.46]{Bergh_Lofstrom}).
\begin{enumerate}
\item $(X_0,X_1)_{\theta,q}=(X_1,X_0)_{1-\theta,q}$

\vspace{1mm}

\item $(X_0,X_1)_{\theta,q}\subset (X_0,X_1)_{\theta,r}$ \ if \ $1\leq q\leq r\leq +\infty$.

\vspace{1mm}

\item If $X_1\subset X_0$ and $\theta_0<\theta_1$, then $(X_0,X_1)_{\theta_1,q}\subset (X_0,X_1)_{\theta_0,q}$.

\end{enumerate}

\begin{theorem}{\rm (Associate space of an interpolation space).}\label{duality theorem}

Let $X_0, X_1$ be two Banach function spaces, such that $X_1\subset X_0$. Let $1\leq q<+\infty, \ \alpha\in \mathbb R, \ 0<\theta<1$.
Then % the associate space of
$$[(X_0,X_1)_{\theta,q;\alpha}]^{'}=(X_1^{'},X_0^{'})_{1-\theta,q';-\alpha}, \ \ \ \frac{1}{q}+\frac{1}{q'}=1,$$
where $X^{'}_i$ is the associate space of $X_i, \ i=0,1$.

\end{theorem}

\vspace{1mm}

\subsection{Identifications of some interpolation spaces}\hfill\\[1mm]
Here we collect some well known and new results concerning the interpolation spaces between Lebesgue, Lorentz, Lorentz-Zygmund, small Lebesgue and $G\Gamma$ spaces.

\vspace{2mm}

Let $0<\theta<1, \ 1\leq q\leq+\infty$. \\[1mm]
We have (see, e.g., \cite[Theorem 5.2.1, Theorem 5.3.1]{Bergh_Lofstrom}, \cite[Theorem 2, p.134]{Triebel_book}, \cite[Theorem 1.9, p.300]{Bennett_Sharpley}) %\cite[p.284]{Peetre1966}  %
\begin{eqnarray}
L^q &=& (L^{p_0},L^{p_1})_{\theta,q}, \ \ \ 1\leq p_0<p_1\leq\infty, \ \ \frac{1}{q}=\frac{1-\theta}{p_0}+\frac{\theta}{p_1}.
\end{eqnarray}
In particular,
\begin{eqnarray}
L^q &=& (L^1,L^\infty)_{1-\frac{1}{q},q}, \ \ \ 1<q<\infty\\
L^q &=& (L^p,L^\infty)_{\theta,q}, \ \ \ 1\leq p<q<\infty, \ \ \ \theta=1-\frac{p}{q}.
\end{eqnarray}
% see Diego Chamorro, Pierre-Gilles Lemarié-Rieusset, Real interpolation method,Lorentz spaces and refined Sobolev inequalities
For $1\leq p_0<p_1\leq\infty, \ 1\leq q_0,q_1\leq \infty, \ \ \displaystyle\frac{1}{p}=\frac{1-\theta}{p_0}+\frac{\theta}{p_1}$,
\begin{eqnarray}\label{identificationLorentz}
L^{p,q} &= & (L^{p_0},L^{p_1})_{\theta,q}=(L^{p_0,q_0},L^{p_1,q_1})_{\theta,q};
\end{eqnarray}
if $p_0=p_1=p$,
\begin{eqnarray}
L^{p,q} &= & (L^{p,q_0},L^{p,q_1})_{\theta,q}, \ \ \ \hbox{with} \ \ \displaystyle\frac{1}{q}=\frac{1-\theta}{q_0}+\frac{\theta}{q_1}.
\end{eqnarray}

\vspace{1mm}

\begin{theorem} \label{Th identifications}{\rm (\textbf{Interpolation with logarithm function between\\ Lorentz spaces})}
{\rm (\cite[p. 908-909]{AFFGER2023})} % \cite[Corollary 5.3, p.99]{Gogatishvili_Opic_Trebels}, \cite[p.20]{Dominguez_Tikhonov}).

\vspace{1mm}

Let $1\leq p_0<p_1\leq \infty$,  \ $1\leq q_0,q_1\leq \infty$, \ $1\leq q<\infty$, \ $\alpha\in \mathbb R$. \\ Let $0< \theta<1$ and $\alpha \in\mathbb R$, or $\alpha\geq - \frac 1q$ and $\theta=0$, or $\theta=1$ and $\alpha<-\frac 1q$.\\[1mm]
 Then the following identifications hold.

\vspace{2mm}

 \textbf{(1.)} \ Case $0<\theta<1$ and $\alpha\in \mathbb R$.

\centerline{$\displaystyle
||f||_{\left(L^{p_0,q_0},\ L^{p_1,q_1}\right)_{\theta,q;\alpha}}\simeq \left(\int_0^1 \left[t^{\frac{1-\theta}{p_0}+\frac{\theta}{p_1}} f_*(t) (1-\log t)^\alpha \right]^q \,\frac{dt}{t}\right)^{\frac{1}{q}},
$}

 which implies
\begin{eqnarray}\label{LZinterpolationLorentz}
\left(L^{p_0,q_0},\ L^{p_1,q_1}\right)_{\theta,q;\alpha}
=L^{p_\theta,q}(\log L)^{\alpha}, \ \ \ \frac{1}{p_\theta}=\frac{1-\theta}{p_0}+\frac{\theta}{p_1}.
\end{eqnarray}
% where $\displaystyle \frac{1}{p_\theta}=\frac{1-\theta}{p_0}+\frac{\theta}{p_1}$.

\vspace{2mm}
As particular cases, for $1\leq r<p<m\leq\infty, \ \ 1\leq q<\infty, \ \alpha\in\mathbb R$, we have
\begin{eqnarray}
L^{p,q}(\log L)^\alpha &= & (L^{r,\infty},L^{m})_{\theta,q;\alpha}, \ \ \ \displaystyle\frac{1}{p}=\frac{1-\theta}{r}+\frac{\theta}{m}, \label{identification_LrinftyLm}\\[1mm]
L^{p,q}(\log L)^\alpha &= & (L^1,L^m)_{\theta,q;\alpha}, \ \ \ \theta=m'\left(1-\frac{1}{p}\right), \ \ m'=\frac{m}{m-1}\label{identification_L_L_LZ_3_parameters}\\[1mm]
L^{p,q}(\log L)^\alpha &= & (L^{r,\infty},L^{\infty})_{\theta,q;\alpha}=(L^{r},L^{\infty})_{\theta,q;\alpha}, \ \ \ \displaystyle \theta=1-\frac{r}{p}\label{identification_LrinftyLinfty}.
\end{eqnarray}

\vspace{3mm}

 \textbf{(2.)} \ Limiting case $\theta=0$, with $\alpha\geq - \dfrac 1q$ and $1\leq q_0<\infty$.
%If  $1\leq q_1<\infty$,
%
%Let $w_1(t)=t^{-1}(1-\log t)^{\lambda p_2}, \ w_2(t)=t^{\frac{q_1}{r}-1}, \ t\in (0,1)$. Then
%\begin{equation}
%\left(L^{r,q_1},\ L^{m,q_2}\right)_{0,p_2;\lambda}
%= G\Gamma(q_1,p_2;w_1,w_2),
%\end{equation}
%%where $w_2(t)=t^{\frac{q_0}{p_0}-1}$ if $q_0<\infty$, \ $w_2(t)=t^{\frac{q_1}{p_1}-1}$ if $q_1<\infty$.
%and, in particular, for $1<m\leq \infty$, $q_1=r=1$, \ $q_2=m$, we have
%\begin{equation}
%(L^1,L^m)_{0,p_2;\lambda}=G\Gamma(1,p_2;w_1).
%\end{equation}
\begin{equation*}
||f||_{\left(L^{p_0,q_0},\ L^{p_1,q_1}\right)_{0,q;\alpha}}\simeq \left(\int_0^1\left[\left(\int_0^t s^{\frac{q_0}{p_0}-1}f_*^{q_0}(s)ds\right)^{\frac{1}{q_0}}(1-\log t)^{\alpha}\right]^q\, \frac{dt}{t}\right)^{\frac{1}{q}},
\end{equation*}
which yields
\begin{eqnarray}\label{LorentLorentz_GGamma}
\left(L^{p_0,q_0},\ L^{p_1,q_1}\right)_{0,q;\alpha}
= G\Gamma(q_0,q;w_1,w_2),
\end{eqnarray}
where $w_1(t)=t^{-1}(1-\log t)^{\alpha q}, \ w_2(t)=t^{\frac{q_0}{p_0}-1}, \ t\in (0,1)$.

\vspace{3mm}

As particular case, for $1<p\leq \infty$, $p_0=q_0=1$, \ $p_1=q_1=p$, we have
\begin{equation}\label{L1LpGGamma}
(L^1,L^p)_{0,q;\alpha}=G\Gamma(1,q;w_1).
\end{equation}

\vspace{2mm}

If  $1<q_0=p_0<p_1<\infty$, \ $1\leq q_1\leq \infty$, \ $q=1$, \  $\alpha>-1$,  \ $\alpha_0=p'_0(\alpha +1)$, \
$\dfrac1{p_0}+\dfrac1{p'_0}=1$, we have the small Lebesgue spaces (see also \cite[Theorem 4.5]{AFH2020_Mediterr})
\begin{equation}
(L^{p_0},L^{p_1,q_1})_{0,1;\alpha}=L^{(p_0,\alpha_0}(\Omega),
\end{equation}
since
\begin{equation*}
||f||_{\left(L^{p_0},\ L^{p_1,q_1}\right)_{0,1;\alpha}}\!\simeq \!\int_0^1\left(\int_0^t f_*^{p_0}(s)ds\right)^{\frac{1}{p_0}}\!(1-\log t)^{\frac{\alpha_0}{p_0^{\prime}}-1}\, \frac{dt}{t}\simeq ||f||_{L^{(p_0,\alpha_0}(\Omega)}.
\end{equation*}

\vspace{2mm}

In the case $q=1$, \ $p_1=q_1$  and $p_0=q_0$, we recover (see \cite[Propositions 4.1 and 4.2, p.437]{FFGKR_NA2018})
$$(L^{p_0},L^{p_1})_{0,1;\frac{\alpha}{p_0^{\prime}}-1}=L^{(p_0,\alpha}(\Omega), \ \ 1<p_0<p_1\leq \infty.$$

\vspace{3mm}

\textbf{(3.)} \ Limiting case $\theta=0$, with $\alpha\geq - \dfrac 1q$ and $ q_0=\infty$.
\begin{equation*}
||f||_{\left(L^{p_0,\infty},\ L^{p_1,q_1}\right)_{0,q;\alpha}}\simeq \left(\int_0^1 \left(\esssup_{0<s<t} s^{\frac{1}{p_0}}\, f_*(s)\right)^q (1-\log t)^{\alpha q}\, \frac{d t}{t}\right)^{\frac{1}{q}}.
\end{equation*}

Therefore, we have
\begin{equation}\label{LorentzLebesguetheta0}
(L^{p_0,\infty},\ L^{p_1,q_1})_{0,q;\alpha}=G\Gamma(\infty, q;w_1,w_2),
\end{equation}
where
$w_1(t)=t^{-1}(1-\log t)^{\alpha q}$, \ $w_2(t)=t^{\frac{1}{p_0}}$, \  $t\in (0,1)$.

\vspace{5mm}

\textbf{(4.)} \ Limiting case $\theta=1$, with $\alpha<-\frac{1}{q}$ and $1\leq q_1<\infty$.
\begin{equation*}
||f||_{\left(L^{p_0,q_0},\ L^{p_1,q_1}\right)_{1,q;\alpha}}\simeq \left(\int_0^1\left(\int_t^1 s^{\frac{q_1}{p_1}-1}f_*^{q_1}(s)ds\right)^{\frac{q}{q_1}}(1-\log t)^{\alpha q}\, \frac{dt}{t}\right)^{\frac{1}{q}}.
\end{equation*}
If, in particular, $q_1=p_1<\infty$, \ $\alpha=\dfrac{\gamma}{q}+\dfrac{\beta}{p_1}$ \ with \ $\gamma>-1$, $\beta\in\mathbb R$, \ $\gamma +\beta \dfrac{q}{p_1}+1<0$, so that $\alpha<-1/q$, we have
\begin{equation*}\label{theta1}
||f||_{\left(L^{p_0,q_0},\ L^{p_1}\right)_{1,q;\alpha}}\simeq \left(\int_0^1\left(\int_t^1 f_*^{p_1}(s)ds\right)^{\frac{q}{p_1}}(1-\log t)^{\gamma +\beta \frac{q}{p_1}}\, \frac{dt}{t}\right)^{\frac{1}{q}}
\end{equation*}
and, from \cite[ Lemma 2.4, p.9]{AFFGR2020},
\begin{equation}\label{Ggammatheta1}
\left(L^{p_0,q_0},\ L^{p_1}\right)_{1,q;\alpha}= G\Gamma(p_1,q;w_1,w_2).
\end{equation}
with $w_1(t)=t^{-1}(1-\log t)^\gamma$, \ $ w_2(t)=(1-\log t)^\beta$.

\vspace{5mm}

\textbf{(5.)} \ Limiting case $\theta=1$, with $\alpha<-\dfrac 1q$ and  $q_1=+\infty$.
\begin{equation*}
||f||_{\left(L^{p_0,q_0},\ L^{p_1,\infty}\right)_{1,q;\alpha}}\simeq \left(\int_0^1 \left(\esssup_{0<s<t} s^{\frac{1}{p_1}}\, f_*(s)\right)^q (1-\log t)^{\alpha q} \, \frac{d t}{t}\right)^{\frac{1}{q}}.
\end{equation*}
Therefore we have
\begin{equation}
(L^{p_0,q_0},\ L^{p_1,\infty})_{1,q;\alpha}=G\Gamma(\infty, q;w_1,w_2) ,
\end{equation}
where $w_1(t)=t^{-1}(1-\log t)^{\alpha q}$, \ $w_2(t)=t^{\frac{1}{p_1}}$, \  $t\in (0,1)$.

\end{theorem}

\vspace{1cm}

The following result can be found in \cite[Lemma 5.5, pag.100]{Gogatishvili_Opic_Trebels}, in the more general setting of the Lorentz-Karamata spaces, of which the Lorentz-Zygmund spaces are a special case.

\begin{theorem}{\rm (\textbf{Interpolation with logarithm function between\\ Lorentz-Zygmund spaces})}

\vspace{1mm}

Let $1\leq p_0< p_1\leq +\infty, \ 1\leq q_0,q_1\leq \infty, \ 1\leq q<\infty$,  $0<\theta<1$, \ $\alpha_0, \alpha_1, \alpha\in \mathbb R$. Then

\begin{equation}
\left(L^{p_0,q_0}(\log L)^{\alpha_0},\ L^{p_1,q_1}(\log L)^{\alpha_1}\right)_{\theta,q;\alpha}
=L^{p_\theta,q}(\log L)^{\alpha_\theta}
\end{equation}
with $\dfrac 1{p_\theta}=\dfrac{1-\theta}{p_0}+\dfrac\theta{p_1}$\; and \; %$\alpha_\theta=(1-\theta)\dfrac{\alpha_0 q}{q_0}+\dfrac{\theta \alpha_1q}{q_1}.$
 $\alpha_\theta=(1-\theta)\alpha_0+\theta \alpha_1+\alpha. $

\vspace{2mm}

In particular, for $\alpha_0=\alpha_1=0$, we recover \eqref{LZinterpolationLorentz} in case \textbf{(1.)}
\begin{eqnarray}
\left(L^{p_0,q_0},\ L^{p_1,q_1}\right)_{\theta,q;\alpha}
=L^{p_\theta,q}(\log L)^{\alpha}.
\end{eqnarray}
\end{theorem}

\vspace{2mm}

\begin{proof}
In \cite{Gogatishvili_Opic_Trebels}, for $\Omega\subset \mathbb R^n$ bounded with $|\Omega|=1$, the Lorentz-Karamata spaces are defined through
\begin{equation}\label{norm Lorentz-Karamata}
\|f\|_{L_{p,q;b}(\Omega)}:=||t^{\frac{1}{p}-\frac{1}{q}}b(t)f_*(t)||_{L^q(0,1)},
\end{equation}
where $b(t)$ is a slowly varying function on $(0,1)$. It has been proved that
\begin{equation}\label{interpolationLorentzKaramata}
\left(L_{p_0,q_0;b_0}, L_{p_1,q_1;b_1}\right)_{\theta,q;b}=L_{p_\theta,q;b_\theta^{\#}}
\end{equation}
where
\begin{equation}
\frac{1}{p_\theta}=\frac{1-\theta}{p_0}+\frac{\theta}{p_1}
\end{equation}
and
\begin{equation}
b_\theta^{\#}(t)=b_0^{1-\theta}(t) \,b_1^\theta(t) \, b\left(t^{\frac{1}{p_0}-\frac{1}{p_1}}\frac{b_0(t)}{b_1(t)}\right), \ \ \hbox{for all} \ \ t\in (0,1).
\end{equation}

 Choosing $b_0(t)=(1-\log t)^{\alpha_0}$, \ $b_1(t)=(1-\log t)^{\alpha_1}$, \ $b(t)=(1-\log t)^\alpha$, we have
$$L_{p_0,q_0;b_0}(\Omega)=L^{p_0,q_0}(\log L)^{\alpha_0}(\Omega), \ \ \ L^{p_1,q_1}(\log L)^{\alpha_1}(\Omega)= L^{p_1,q_1}(\log L)^{\alpha_1}(\Omega) .$$
Therefore, the left hand side of \eqref{interpolationLorentzKaramata} becomes
$$
\left(L^{p_0,q_0}(\log L)^{\alpha_0}, L^{p_1,q_1}(\log L)^{\alpha_1}\right)_{\theta,q;\alpha}.
$$
Now, to identify the space $L_{p_\theta,q;b_\theta^{\#}}$, we determine $b_\theta^{\#}$:
\begin{equation*}
b_\theta^{\#}(t)=(1-\log t)^{\alpha_0 (1-\theta)+\alpha_1 \theta}[1-\log(t^{\frac{1}{p_0}-\frac{1}{p_1}}(1-\log t)^{\alpha_0-\alpha_1})]^\alpha.
\end{equation*}
It is easy to see that
\begin{equation*}
\lim_{t\to 0^+} \frac{1-\log(t^{\frac{1}{p_0}-\frac{1}{p_1}}(1-\log t)^{\alpha_0-\alpha_1})}{1-\log t}=\frac{1}{p_0}-\frac{1}{p_1}
\end{equation*}
since
%\begin{eqnarray*}
%1-\log(t^{\frac{1}{p_0}-\frac{1}{p_1}}(1-\log t)^{\alpha_0-\alpha_1})=1-\left(\frac{1}{p_0}-\frac{1}{p_1}\right)\log t-(\alpha_0-\alpha_1)\log(1-\log t),
%\end{eqnarray*}
$$
 \lim_{t\to 0^+}\frac{1-\left(\frac{1}{p_0}-\frac{1}{p_1}\right)\log t}{1-\log t}=\frac{1}{p_0}-\frac{1}{p_1}, \ \ \lim_{t\to 0^+} \frac{(\alpha_0-\alpha_1)\log(1-\log t)}{1-\log t}=0.
$$
Therefore
$$
b_\theta^{\#}(t)\simeq (1-\log t)^{\alpha_0 (1-\theta)+\alpha_1 \theta}(1-\log t)^\alpha=(1-\log t)^{\alpha_0 (1-\theta)+\alpha_1 \theta+\alpha},
$$
hence
$$
L_{p_\theta,q;b_\theta^{\#}}=L^{p_\theta,q}(\log L)^{\alpha_0 (1-\theta)+\alpha_1 \theta+\alpha}=L^{p_\theta,q}(\log L)^{\alpha_\theta}.
$$

\end{proof}

\vspace{1mm}

In the limiting case $p_0=p_1$ and $q_0=q_1$, the following description of interpolation spaces between Lorentz–Zygmund spaces is given in \cite[(7.2), p.21]{AFF2022_JFAA}, where $G\Gamma$ spaces are defined in slightly different way, but that here we rewrite according our notations.

\vspace{2mm}

Let $0<p,q,r<\infty$, \ $0<\theta<1$,  \ $\alpha,\beta\in\mathbb R, \ \alpha<\beta$. Let
$$
w_1(t)=t^{-1}(1-\ln t)^{\theta r(\beta-\alpha)-1}, \ \ \ \ w_2(t)=t^{q/p-1}(1-\ln t)^{\alpha q}, \quad 0<t< 1. %v(t)=(1-\ln t)^{\beta-\alpha-1/q},
$$
Then
\begin{equation}
\left(L^{p,q}(\log L)^{\alpha}, L^{p,q}(\log
L)^{\beta}\right)_{\theta,r}=G\Gamma(q,r;w_1, w_2).
\end{equation}

\vspace{2mm}

% per mettere formula in section title \texorpdfstring{$R^2$}{R\texttwosuperior}

\section{Interpolation of H{\"o}lderian mappings} \label{Section_Holderian_mappings}

Let $0<\alpha\leq 1$ and $X, Y$ two normed spaces.
A mapping

\vspace{2mm}

\centerline{${\mathcal T}:X\to Y$}

\vspace{2mm}

 is globally $\alpha$-H\"olderian, with H\"older constant $c$, if
$$
\exists \, c>0 \ : \ ||{\mathcal T}a-{\mathcal T}b||_Y\leq c ||a-b||^{\alpha}_X, \ \ \forall\, a,b\in X.
$$

\vspace{2mm}

In the sequel we will consider four normed spaces $X_1\subset X_0$, $Y_1\subset Y_0$ and
$${\mathcal T}: X_i\to Y_i, \ \ i=0,1,$$
 a nonlinear mapping satisfying, for $0<\alpha\leq 1$ and $\beta>0$, the following conditions:
\begin{enumerate}
\item \label{1mapT} $||{\mathcal T}a-{\mathcal T}b||_{Y_0}\leq f(||a||_{X_0},||b||_{X_0})||a-b||^\alpha_{X_0}, \ \ \forall\, a,b \in X_0$,

\vspace{2mm}

\item \label{2mapT} $||{\mathcal T}a||_{Y_1}\leq g(||a||_{X_0})||a||_{X_1}^\beta,\ \ \forall\,a\in X_1,$
\end{enumerate}

\vspace{2mm}

where $g$ is a continuous increasing function on $\mathbb R_+$, $f$ is continuous on $\mathbb R_+^2$ and increasing for any $\sigma\in \mathbb R_+$.

\vspace{2mm}

We put
 \begin{equation}\label{Gmax}
 G(\sigma)=\max\{g(2\sigma);f(\sigma,2\sigma)\}, \ \ \sigma\in\mathbb R_+.
 \end{equation}

\vspace{2mm}

\noindent In the next Theorems we extend some Tartar's results contained in \cite{Tartar1972BullFrance,Tartar1972}.

%Theorem 2.2 in AFFGHR_Anal_Math

\begin{theorem}\label{AFFGHR_Th2-2}
Let $\lambda\in \mathbb R$ and $0<\alpha\leq 1$. Let $X_1\subset X_0$, $Y_1 \subset Y_0$ be normed spaces and assume that $X_1$ is dense in $X_0$. If
$${\mathcal T}: X_i\to Y_i, \ \ \hbox{is globally} \ \ \alpha\hbox{-H\"olderian}, \  \ i=0,1,$$
then, for $0\leq\theta\leq 1$ and $1\leq p\leq+\infty$,
$$ {\mathcal T}: (X_0,X_1)_{\theta,p;\lambda} \to (Y_0,Y_1)_{\theta,\frac{p}{\alpha};\lambda \alpha} \ \ \hbox{is} \ \  \alpha\hbox{-H\"olderian}, $$

\vspace{3mm}
\centerline{ i.e.  \ \ $ \displaystyle ||{\mathcal T}a-{\mathcal T}b||_{\theta,\frac{p}{\alpha};\lambda \alpha} \leq c\, ||a-b||^\alpha_{\theta,p;\lambda}, \ \ \ \forall\, a,b \in (X_0,X_1)_{\theta,p;\lambda}
$}

\end{theorem}

\vspace{2mm}

%Theorem 2.1 in AFFGHR_Anal_Math

\begin{theorem}\label{AFFGHR_Th2-1}
Let $\lambda\in \mathbb R$, $0<\alpha\leq 1$ and $\beta>0$, with $\alpha\leq\beta$. Let $X_1\subset X_0$, $Y_1 \subset Y_0$ be normed spaces and
$${\mathcal T}: X_i\to Y_i, \ \ i=0,1,$$
a mapping satisfying the assumptions \eqref{1mapT} and \eqref{2mapT}.%, where $f$ and $g$ are continuous functions on $\mathbb R_+^2$ and $\mathbb R_+$, respectively, $g$ is increasing,  $f(\sigma;\cdot)$ is increasing for any $\sigma\in \mathbb R_+$.

Then, for $0\leq\theta\leq 1$ and $1\leq p\leq+\infty$,
$${\mathcal T} \ \ \hbox{maps} \ \  (X_0,X_1)_{\theta,p;\lambda} \ \ \hbox{into} \ \ (Y_0,Y_1)_{\theta\frac\alpha\beta,\frac p\alpha;\lambda\alpha}$$
and
\begin{equation*}
||{\mathcal T}a||_{\theta\frac\alpha\beta,\frac p\alpha;\lambda \alpha}\lesssim
[(1+||a||_{X_0}^{\beta-\alpha})\,G(||a||_{X_0})]\,||a||^\alpha_{\theta,p;\lambda}.
\end{equation*}

\end{theorem}

\vspace{2mm}

\section{Estimates of the \texorpdfstring{$K$}{R\texttwosuperior}
-functional related to the mapping \texorpdfstring{$\mathcal T$}{R\texttwosuperior}}\label{SecKfunctional}

Here we study the action of the nonlinear mappings $\mathcal T$ on the $K$-funtional, defined in Section \ref{Section_interpolation_spaces} (see \cite[Section 2.1]{AFFGER2023}.

\begin{theorem}
Let $0<\alpha\leq 1$, \ $\beta>0$. Let $X_1\subset X_0$, $Y_1
\subset Y_0$, be four normed spaces and
$${\mathcal T}: X_i\to Y_i, \ \ i=0,1, $$
a nonlinear mapping satisfying the assumptions \eqref{1mapT} and \eqref{2mapT} in Section \ref{Section_Holderian_mappings} and the function $G$ defined in \eqref{Gmax}. Then, $\forall\,a\in X_0$, \ $\forall\,t>0$, we have
% Let $G(\sigma)=\max(g(2\sigma);f(\sigma;2\sigma)),\sigma\in\mathbb R_+.$ \ Then, $\forall\,a\in X_0$, \ $\forall\,t>0$, we have
$$K\big({\mathcal T}a,t^\beta,Y_0,Y_1\big)=K({\mathcal T}a,t^\beta)\leq G(||a||_{X_0})\left([K(a,t)]^\beta+[K(a,t)]^\alpha\right).$$
%\vspace{2mm}
Moreover, if $\beta\geq \alpha $, \ then
$$K({\mathcal T}a,t^\beta)\leq G(||a||_{X_0})(1+||a||_{X_0}^{\beta-\alpha})[K(a,t)]^\alpha. $$

\end{theorem}

As a particular case we have the following consequences.

\begin{corollary}

Let $0<\alpha\leq 1$, \ $\beta>0$. Let $X_1\subset X_0$, $Y_1
\subset Y_0$, be four normed spaces. Assume that the mapping ${\mathcal T}: X_0\to Y_0 $  is globally $\alpha$-H\"olderian with constant $M_0$, i.e.
$$||{\mathcal T}a-{\mathcal T}b||_{Y_0}\leq M_0 ||a-b||^{\alpha}_{X_0}, \ \ \ \forall\, a,b\in X_0,$$
and ${\mathcal T}$ maps $X_1$ into $Y_1$, in the sense that $\exists\,M_1>0, \ \beta>0$ \ s.t. $$||{\mathcal T}a||_{Y_1} \leq M_1 ||a||_{X_1}^\beta.$$
Then
$$K({\mathcal T}a,t^\beta)\leq\max(M_0;M_1)\Big([K(a,t)]^\beta+[K(a,t)]^\alpha\Big), \ \ \forall\,a\in X_0,\ \forall\,t>0.$$
Moreover, if $\beta\geq \alpha $, then
$$K({\mathcal T}a,t^\beta)\leq \max(M_0;M_1)\,(1+||a||_{X_0}^{\beta-\alpha})\,[K(a,t)]^\alpha.$$

\end{corollary}

\vspace{2mm}

\begin{corollary}
Let $0<\alpha\leq 1$. Let $X_1\subset X_0$, $Y_1
\subset Y_0$, be four normed spaces. Assume that the mapping
$${\mathcal T}: X_i\to Y_i, \ \ \ i=0,1 , $$
is globally $\alpha$-H\"olderian with constant $M_i$, for $i=0,1 $, \  i.e.
$$||{\mathcal T}a-{\mathcal T}b||_{Y_i}\leq M_i ||a-b||^{\alpha}_{X_i}, \ \ \forall\, a,b\in X_i.$$
Then, $\forall\,a\in X_0,\ \forall\,b\in X_1, \ \forall\, t>0$, we have
$$K({\mathcal T}a-{\mathcal T}b,t^\alpha)\leq 2 \max(M_0;M_1)\,[K(a-b,t)]^\alpha.$$
Furthermore, if $X_1$ is dense in $X_0$, then the above inequality holds also for all $b\in X_0$.

\end{corollary}

\vspace{2mm}

\section{Applications to the regularity of the solution of a \texorpdfstring{$p$}{R\texttwosuperior}-Laplacian equation}\label{SecPDEs}

The Marcinkiewicz interpolation theorems for linear operators acting on Lebesgue spaces turned out to be a powerful tool for studying regularity of solutions for linear PDEs in $L^p$-spaces.
The $K$-method introduced by J. Peetre (\cite{Peetre1968, Peetre1970_StudiaMath}) allowed to extend the study of  regularity of solutions of linear equations on spaces different from  $L^p$-spaces. The main difficulty to apply Peetre's definition is the identification of the interpolation spaces between two normed spaces embedded in a same topological space. In \cite{AFFGR2020,FFGKR_NA2018,FFR-DIE2017,Fiorenza-Formica-Gogatishvili-DEA2018} we did such a study with applications to linear PDEs using non-standard spaces as grand or small Lebesgue spaces and $G\Gamma$-spaces. \par
In \cite{Tartar1972} L. Tartar gave interpolation results on nonlinear H\"olderian mappings (which include Lipschitz mappings) with applications to semi-linear PDEs.

Other results concerning interpolation of Lipschitz operators and other applications of
Interpolation theory, also in PDEs, can be found, e.g., in \cite{Cianchi-Mazja_European,Maligranda1984Studia,Maligranda1989,Maligranda1990bibliography,Maligranda_Persson_Wyller,Peetre1970}.

\vspace{2mm}

\subsection{The quasilinear \texorpdfstring{$p$}{R\texttwosuperior}-Laplacian equation}\label{subsecEquation}\hfill\\[1mm]
In the sequel we will consider $\Omega$ a bounded smooth domain of $\mathbb R^n$, \ $n\geq 2$, \ $1<p<+\infty$, \  $ \dfrac 1p+\dfrac 1{p'}=1$. Recall that $W^{1,p}(\Omega)$ is the classical Sobolev space of all real-valued functions $f\in L^p(\Omega)$ whose first-order weak (or distributional) partial derivatives on $\Omega$ belong to $L^p(\Omega)$, normed by
$$
||f||_{W^{1,p}(\Omega)}=||f||_{L^p(\Omega)}+||\nabla f||_{L^p(\Omega)}.
$$
The space $W_0^{1,p}(\Omega)$ denotes the closure of $C_0^\infty(\Omega)$ in $W^{1,p}(\Omega)$; rougly speaking $W_0^{1,p}(\Omega)$ is a subspace of $W^{1,p}(\Omega)$ consisting of functions which vanish on the boundary of $\Omega$. Its dual is $W^{-1,p'}(\Omega)=(W_0^{1,p}(\Omega))'$.

\vspace{2mm}

We provide applications of the results described in Sections \ref{Section_interpolation_spaces}, \ref{Section_Holderian_mappings} and \ref{SecKfunctional} to regularity on the gradient of the weak or entropic-renormalized solution $u$ to the quasilinear equation of the form
\begin{equation}\label{divEq}
 -\div(|\nabla u|^{p-2}\nabla u)+V(x;u)=f,  \ \ \ \ u=0 \ {\rm on} \ \partial \Omega,
 \end{equation}
associated to the Dirichlet homogeneous condition on the boundary, where $1< p<\infty$, $V$ is a nonlinear potential and  $f$ belongs to non-standard spaces such as Lorentz-Zygmund spaces.

%We also show that the mapping ${\mathcal T}: \ {\mathcal T}f=\nabla u$ is locally or globally $\alpha$-H\"olderian under suitable values of $\alpha$ and appropriate assumptions on $V$.

\vspace{2mm}
More precisely, we assume that
$V:\Omega\times \mathbb R \to\mathbb R$ is a Carath{é}odory function satisfying:

\begin{enumerate}
\renewcommand{\labelenumi}{(\theenumi)}
\item[(H1)\label{nH1}] the mapping $x\in\Omega\to V(x;\sigma)$ is in $L^\infty(\Omega)$  for every $\sigma\in\mathbb R$;
\item[(H2)\label{nH2}] the mapping $\sigma\in\mathbb R\to V(x;\sigma)$ is continuous and
non decreasing for almost every $x\in \Omega$ and $V(x;0)=0$.
\end{enumerate}
%The datum $f$ of the equation will be assumed in non-standard spaces, such as Lorentz-Zygmund spaces.

\vspace{2mm}

The $p$-laplacian is \emph{strong coercive}, that is satisfies the following inequality, which can be found in \cite[Lemma 4.10, p.264]{Diaz_book_1985}.

For $p\geq 2$, there exists a constant $\alpha_p>0$ such that, for all $\xi,\xi'\in \mathbb R^n$,
\begin{equation}\label{ineq_coercivity}
(|\xi|^{p-2}\xi -|\xi'|^{p-2}\xi',\xi-\xi')_{\mathbb R^n}\geq \alpha_p |\xi-\xi'|^p,
\end{equation}
where the symbol $(\cdot, \cdot)_{\mathbb R^n}$ in the left-hand side denotes the inner product in
$\mathbb R^n$ and $|\cdot|$ is the associated norm.

\vspace{2mm}
The next Proposition follows from Leray–Lions'method for monotone operators (see \cite{Lions1969book}) or the usual fixed point theorem of Leray–Schauder's type (see \cite{Gilbarg_Trudinger}).

\begin{proposition}{\rm (\textbf{weak solution})} \label{defweaksol}\\
Let $f\in L^1(\Omega)\cap W^{-1,p'}(\Omega)$, \ $1<p<+\infty$, \  $ \dfrac 1p+\dfrac 1{p'}=1$.

Then there exists a unique \textbf{weak solution} of the Dirichlet problem \eqref{divEq},
%\begin{equation*}\label{p-laplacian-eq}
%-\div(|\nabla u|^{p-2}\nabla u)+V(x;u)=f,
%\end{equation*}
i.e. a unique  $u\in W^{1,p}_0(\Omega)$ such that
\begin{equation}\label{weak sol}
 \begin{gathered}
 \int_\Omega\varphi(x)V(x;u)dx+\int_\Omega|\nabla u|^{p-2}\nabla u\cdot \nabla \varphi dx=\int_\Omega f\varphi\,dx ,\\
 \ \ \ \forall\,\varphi\in W^{1,p}_0(\Omega)\cap L^\infty(\Omega)
\end{gathered}
\end{equation}

\end{proposition}

\begin{remark}
We observe that if $p>n$ then $L^1(\Omega)\subset W^{-1,p'}(\Omega)$. In this case it is well known that, if the datum $f$ of \eqref{divEq} is in $L^1(\Omega)$, there is existence and uniqueness of the weak solution, which is bounded (thanks to the Sobolev embedding). The existence follows from the classical results on operators acting between Sobolev spaces in duality (see, e.g., \cite[p.107]{Lions1969book}, \cite{BoccardoGalluet,Dallaglio,LerayLions}).

\noindent If $1<p<n$ then $L^{(p^*)'}(\Omega)\subset W^{-1,p'}(\Omega)$, where $p^*=\frac{np}{n-p}$, and also $L^{(p^*)'}(\Omega)\subset L^1(\Omega)$. Therefore, if the datum $f$ of \eqref{divEq} is in $L^{(p^*)'}(\Omega)$, we have $f\in L^1(\Omega)\cap W^{-1,p'}(\Omega)$, hence Proposition \ref{defweaksol} can be applied.

\noindent If $p=n$ then $L^{q'}(\Omega)\subset W^{-1,p'}(\Omega)$ for any $q\in[1,\infty[$, where $q'=\frac{q}{q-1}$. Therefore, if $f\in L^{q'}(\Omega)$, for the same above reasons, we can apply Proposition \ref{defweaksol}.

\end{remark}

\begin{remark}
In the case $p=n$ and $f$ in $L^1(\Omega)$ the Iwaniec-Sbordone's method guarantees existence and uniqueness of the weak solution of \eqref{divEq}; see, e.g., \cite{Greco-Iwaniec-Sbordone}, where the authors use the Hodge decomposition in a smart way in a variety of questions (see also \cite{Ferone_Jalal,Miranville2001}); see \cite{Fiorenza_Sbordone} for the particular case $p=n=2$).
\end{remark}

For the above Remark, the meaningful case is
$$p<n,$$
even if all our next results remain still true for $p\geq n$.

\vspace{2mm}

We define a nonlinear mapping
\begin{equation} \label{defT}
\begin{matrix}
{\mathcal T}:&L^1(\Omega)\cap W^{-1,p'}(\Omega) &\longrightarrow& \big[L^p(\Omega)\big]^n\\[3mm]
&f&\longmapsto&{\mathcal T}f=\nabla u .
\end{matrix}
\end{equation}
We intend to extend the mapping $\mathcal T$ over all $L^1(\Omega)$.

\vspace{2mm}

If $p<n$ and $f$ is only in $L^1(\Omega)$, the formulation by equation \eqref{weak sol} cannot ensure the uniqueness of the solution.

\vspace{2mm}

This case attracted the interest of several researchers, who tried to find a satisfying notion of solution in order to get both existence and uniqueness of the solution (see, for instance, \cite{Benilan1995,Blanchard_Murat,Boccardo_Diaz_Giachetti_Murat,Carrillo_Wittbold,DiPerna_Lions1989,
Rakotoson1993DIE,Rakotoson1996proprietes}).

Here we focus our attention to the so-called entropic-renormalized solutions, considered by Jean Michel Rakotoson in \cite{Rakotoson1994uniqueness,Rakotoson1996proprietes} (see also \cite{Benilan1995}), which are defined as follows.

\vspace{2mm}

\begin{definition} (\textbf{Entropic-renormalized solution})\label{entropic_sol}

Let $\Omega \subset \mathbb R^n$ be a bounded smooth domain.

For all $k>0$, we consider the truncation operator $T_k:\mathbb R\to \mathbb R$ defined by
\begin{equation}\label{truncation}
T_k(\sigma)=\{|k+\sigma|-|k-\sigma|\}/2,
\end{equation}
and we define $\mathbb S^{1,p}_0$ as the set of all measurable functions $v:\Omega\to\mathbb R $ satisfying:

\vspace{2mm}

{\textbf{1.}} \ $ {\rm tan}^{-1}(v)\in W^{1,1}_0(\Omega)$;

\vspace{2mm}

{\textbf{2.}} \ $\forall\,k>0, \ T_k(v)\in W^{1,p}_0(\Omega)$;

\vspace{2mm}

{\textbf{3.}} \ $\displaystyle\sup_{k>0}k^{-\frac1p}||\nabla T_k(v)||_{L^p(\Omega)}<\infty$.

\end{definition}
 %(see \cite{Rakotoson1994uniqueness,Rakotoson1996proprietes,Benilan1995}).

\vspace{2mm}

A function $u$ defined on $\Omega$ is an entropic-renormalized solution of the Dirichlet problem
\begin{equation}\label{Eq-datumL1}
-\Delta_pu+V(x;u)
=f\in L^1(\Omega),\quad\quad  u=0 \quad {\rm on} \ \partial\Omega
\end{equation}
if

\vspace{2mm}

(1) \  $u\in\mathbb S^{1,p}_0(\Omega)$, $V(\cdot, u)\in L^1(\Omega).$

\vspace{2mm}

(2) \  $\forall\,\eta\in W^{1,r}(\Omega),\ r>n,\ \ \forall\,\varphi\in W^{1,p}_0(\Omega)\cap L^{\infty}(\Omega)$ and all $B\in W^{1,\infty}(\mathbb R)$ with $ B(0)=0$, $B'(\sigma)=0$ for all $\sigma$ such $|\sigma|\geq \sigma_0>0$, one has:
\begin{equation}\label{renormalized sol}
\!\int_\Omega|\nabla u|^{p-2}\nabla u\cdot\nabla\Big(\eta B(u-\varphi)\Big)dx+
\int_\Omega V(x;u)\eta  B(u-\varphi)dx=\!\int_\Omega f\eta B(u-\varphi)dx
\end{equation}

\begin{remark}\label{entropic-weak}
If $f\in L^{p'}(\Omega)$ then the formulation \eqref{renormalized sol} is equivalent to the formulation \eqref{weak sol}, i.e. a weak solution is an entropic-renormalized solution (see \cite{Rakotoson1994uniqueness} or \cite{Miranville2001} for the case $p=n$). In
\cite{Benilan1995,Blanchard_Murat,Carrillo_Wittbold,Rakotoson1993DIE,Rakotoson2000_Napoli} existence and uniqueness of an entropic-renormalized solution have been proved.

\end{remark}

\begin{theorem}\label{thexistenceentropic}

Let $f\in L^1(\Omega)$ and
$V:\Omega\times \mathbb R \to\mathbb R$ a Caratheodory function satisfying the assumptions (H1) and (H2).
%\begin{description}
% \item[(H1)] \ for all $\sigma\in\mathbb R, \ x\in\Omega\to V(x;\sigma)$ is in $L^\infty(\Omega)$;
%\item[(H2)] \ for a.e. $x\in \Omega, \ \sigma\in\mathbb R\to V(x;\sigma)$ is continuous and %\\
%%\hspace{0.62cm}
%non decreasing, $V(x;0)=0.$
%\end{description}
%\vspace{1mm}
 Then there exists a unique entropic-renormalized solution of the equation \eqref{Eq-datumL1}.

%\vspace{2mm}

%\centerline{
%\color{red}{$-\Delta_pu+V(x;u)
%=f\in L^1(\Omega),\quad\quad  u=0 \quad {\rm on} \ \partial\Omega.
%$}
%}

%\vspace{1mm}
 Moreover, let $u_j$ be a sequence of the solutions of \eqref{Eq-datumL1} with respective data $f_j$, $j\in \mathbb N$. If the sequence $(f_j)$ converges to $f$ in $L^1(\Omega)$,  then the sequence $(\nabla u_j(x))$ converges to $\nabla u(x)$ almost everywhere in $\Omega$ (up to subsequences still denoted by $\{u_j\}$).
 %\vspace{1mm}
When $p>2-\dfrac 1n$, the solution $u\in  W^{1,1}_0(\Omega).$

\end{theorem}

\vspace{1mm}

The next Proposition gathers well known results (see \cite{Cianchi-Mazja_Arch,Cianchi-Mazja_European,Ferone_Posteraro_Rakotoson,Rakotoson2008rearrangement,
Rakotoson2021HAL}).

\vspace{1mm}

We shall need the following additional growth assumption on $V$.

\vspace{2mm}

(H3) \  There exist a constant $c>0$ and $m_1\in [p-1, \overline{m}_1[$, where
$$
\overline {m}_1=\begin{cases}(p-1)\left(1+\dfrac1{n-p}\right)&if\ p<n\\
<+\infty&if\ p\geq n,\end{cases}
$$
 such that
$$|V(x,\sigma)|\leq c|\sigma|^{m_1},\qquad\forall\,\,\sigma\in\mathbb R,\ a.e.\ x\in\Omega.
$$

\vspace{2mm}

\begin{proposition}\label{prop_u_bounded}
Let $u$ be the solution of the Dirichlet problem \eqref{Eq-datumL1}, with $f\in L^{p'}(\Omega)$. Let us assume that $V$ satisfies (H1) and (H2).

\begin{itemize}

\item If $f\in L^{\frac{n}{p},\frac{1}{p-1}}(\Omega)$ and $p\leq n$, then $u\in L^\infty(\Omega)$ and
\begin{eqnarray}\label{first}
||u||_{L^\infty(\Omega)}\leq c||f||^{\frac{1}{p-1}}_{L^{\frac{n}{p},\frac{1}{p-1}}(\Omega)}.
\end{eqnarray}

\item If $f\in L^{1,\frac{1}{p-1}}(\Omega)$ and $p> n$, then $u\in L^\infty(\Omega)$ and
\begin{eqnarray}\label{second}
||u||_{L^\infty(\Omega)}\leq c||f||^{\frac{1}{p-1}}_{L^{1,\frac{1}{p-1}}(\Omega)}.
\end{eqnarray}

\item If $V$ satisfies also the growth assumption (H3) and $f \in L^{n,1}(\Omega)$, $n\geq 3$, then
\begin{equation}\label{gradientbounded}
||\nabla u||_{L^\infty(\Omega)}\leq c\Big(1+||f||^{\frac{m_1+1-p}{p-1}}_{L^1}\Big)||f||^{\frac1{p-1}}_{L^{n,1}(\Omega)}.
\end{equation}
All the constants denoted by $c$ depend only on $p,\Omega, V$.

\end{itemize}

\end{proposition}

\vspace{2mm}

\begin{remark}
We observe that, for $p>n \ \Rightarrow \ L^{p'}\subset L^{1,\frac{1}{p-1}}$.

\noindent For $p\geq 2\ \Rightarrow \ L^{n,1}\subset L^{p'}$ since, if $p>n$, we have $p'<n'\leq n$; if $2\leq p\leq n$, then $\frac{p}{p-1}\leq p\leq n$, hence $p'\leq n$.
\end{remark}

\vspace{2mm}

The following Lemma provides estimates for the potential $V$.

\begin{lemma}\label{estimatesV}

Let $u$ be a weak solution of \eqref{divEq}:
$$-\Delta_p u +V(x;u)=f, \quad\quad  u=0 \quad {\rm on} \ \partial\Omega.$$

\vspace{2mm}

Assume that $V$ satisfies (H1), (H2), (H3). Let $m_1$ the number defined in condition (H3).

\vspace{2mm}

Let $m_3=\displaystyle \frac{n}{n-p}(p-1)$ \ if\ $1<p<n$ \ or \ $m_3\in[nm_1,+\infty[$ \ if \ $p\geq n$.

\vspace{2mm}

\begin{itemize}

\item If \ $u\in L^\infty(\Omega)$, then
\begin{equation}\label{estimateV_u}
||V(\cdot;u(\cdot))||_{L^{n,1}}^{\frac1{p-1}}\lesssim||u||_\infty\cdot
||u||_{L^{m_3,1}}^{\frac{m_1+1-p}{p-1}}.
\end{equation}

\item If \ $f\in L^{n,1}(\Omega)$, then
\begin{equation}\label{estimateV_f}
||V(\cdot;u(\cdot)||^{\frac1{p-1}}_{L^{n,1}}\lesssim ||f||_{L^{n,1}}^\frac1{p-1}\cdot||f||_{L^1}^{\frac{m_1+1-p}{p-1}}.
\end{equation}

\end{itemize}

\end{lemma}

\vspace{2mm}

We point out that the proof of \eqref{gradientbounded} in Proposition \ref{prop_u_bounded} is consequence of Cianchi-Maz'ya estimate and \eqref{estimateV_f} in Lemma \ref{estimatesV}. In fact, if $u$ is a weak solution of \eqref{divEq} for $f\in L^{n,1}(\Omega)$, then  $-\Delta_p u=f-V(\cdot; u)\in L^{n,1}$. By Cianchi-Maz'ya result (see \cite[Theorem 4.4]{Cianchi-Mazja_European}) and our estimate \eqref{estimateV_f}, we have
\begin{eqnarray*}
||\nabla u||_{L^\infty} &\lesssim &||f-V(\cdot;u)||^{\frac1{p-1}}_{L^{n,1}}\lesssim
||f||_{L^{n,1}}^{\frac1{p-1}}+||V(\cdot;u)||_{L^{n,1}}^{\frac1{p-1}}\\
& \lesssim & ||f||_{L^{n,1}}^{\frac1{p-1}}+||f||_{L^{n,1}}^\frac1{p-1}\cdot||f||_{L^1}^{\frac{m_1+1-p}{p-1}}\\
& \lesssim & \Big(1+||f||^{\frac{m_1+1-p}{p-1}}_{L^1}\Big)||f||^{\frac1{p-1}}_{L^{n,1}(\Omega)}.
\end{eqnarray*}

\vspace{4mm}

\section{The H\"olderian mappings for the case \texorpdfstring{$p\geq 2$}{R\texttwosuperior}}\label{SecApplHolderianpgreater2}

In this Section we apply our previuos results to H\"olderian mappings in order to obtain regularity results on the gradient of the solution of the equation \eqref{Eq-datumL1}:
\begin{equation*}
 -\div(|\nabla u|^{p-2}\nabla u)+V(x;u)=f\in L^1,  \ \ \ \ u=0 \ {\rm on} \ \partial \Omega.
 \end{equation*}

\vspace{1mm}
We recall again that, even if the next results remain valid in the case $p\geq n$, we will consider only the meaningful case
$$p<n.$$
For the case $p\geq n$, the number $p^*=\frac{np}{n-p}$ appearing in the case $p<n$, should be replaced by any finite number.

Recall also that
$$p'=\frac{p}{p-1}, \ \ \   n'=\frac{n}{n-1}, \ \ \  (p^*)' = \frac{np}{np-n+p}.$$

\begin{theorem}\label{th_estimate_gradient}
Let $2\leq p<n$, $f_1,f_2\in L^1(\Omega)$ and $u_1,u_2$ be the corresponding entropic-renormalized solution of \eqref{Eq-datumL1}. Then
\begin{enumerate}

\item \ $\displaystyle \int_\Omega |\nabla T_k(u_1-u_2)|^p\,dx\leq c \int_\Omega |f_1-f_2|\,dx$,

\vspace{1mm}

\item \ $ u_1,u_2\in W_0^{1,1}(\Omega)$ \ and \ $\displaystyle ||\nabla(u_1-u_2)||_{L^{n'(p-1),\infty}(\Omega)}\leq c ||f_1-f_2||^{\frac{1}{p-1}}_{L^1(\Omega)}$,

\end{enumerate}

\vspace{1mm}
where  $n'=\frac{n}{n-1}$ and $c$ is a constant depending only on $p,\Omega, V$.

\end{theorem}

As a consequence we have

\begin{corollary}\label{cor_extensionT}{\rm (of Theorem \ref{th_estimate_gradient})}

 Assume (H1) and (H2) and let $u$ be the unique entropic-renormalized solution of the Dirichlet problem \eqref{Eq-datumL1}. Let ${\mathcal T}$ be the mapping $f\mapsto {\mathcal T} f$, with
$${\mathcal T}f=\nabla u.$$
For $2\leq p<n$ we extend the mapping
$$
{\mathcal T}: L^1(\Omega)\longrightarrow [L^{n'(p-1),\infty}(\Omega)]^n
$$
and  ${\mathcal T}$ is \ $\frac{1}{p-1}$-H\"olderian, i.e.
$$
 \exists \, c(p,\Omega)>0 \ : \ ||{\mathcal T}f_1-{\mathcal T}f_2||_{L^{n'(p-1),\infty}}\leq c(p,\Omega)||f_1-f_2||^{\frac1{p-1}}_{L^1(\Omega)},
$$
where $n'=\frac{n}{n-1}$.
\end{corollary}

\begin{remark}
If $2\leq p<n$ then $L^p(\Omega)\subset L^{n'(p-1),\infty}(\Omega)$, since $n'<p'$  and hence $n'(p-1)<p$.
\end{remark}

\vspace{2mm}

\begin{theorem}\label{ThTholerianp*'p}

Assume (H1) and (H2) and let $u$ be the unique entropic-renormalized solution of the Dirichlet problem \eqref{Eq-datumL1}. Let ${\mathcal T}$ be the mapping $f\mapsto {\mathcal T} f$, with
$${\mathcal T}f=\nabla u.$$
For $2\leq p<n$ the previous mapping in Corollary \ref{cor_extensionT} is also $\frac{1}{p-1}$-H\"olderian from $L^{(p^*)'} $ into  $ [L^p(\Omega)]^n$,
$$
{\mathcal T}: L^{(p^*)'} \to [L^p(\Omega)]^n,
$$
$$ {\rm i.e.} \  \exists \, c_p>0 \ : \ ||{\mathcal T}f_1-{\mathcal T}f_2||_{L^{p}}\leq c_p||f_1-f_2||^{\frac1{p-1}}_{L^{(p^*)'}},$$
where
$p'=\frac{p}{p-1}, \ \ \   p^*=\frac{np}{n-p}, \ \ \  (p^*)' = \frac{np}{np-n+p}$. %\frac1{(p^*)'}=\frac1{p'}+\frac1n$.
\end{theorem}

\begin{proof}
The proof is based on the strong coercivity of the $p$-laplacian and the Poincar\'e-Sobolev inequality. In fact, for $p\geq 2$, we recall that, from \eqref{ineq_coercivity}, there exists a constant $\alpha_p>0$ such that, $\forall\,\xi\in\mathbb R^n$, \  $\forall\,\xi'\in\mathbb R^n$, \begin{equation*}
\Big(|\xi|^{p-2}\xi-|\xi'|^{p-2}\xi',\xi-\xi'\Big)_{\mathbb R^n}\geq \alpha_p|\xi-\xi'|^p.
\end{equation*}
Therefore, for two data $f_1$ and $f_2$ in $L^{p'}(\Omega)$, dropping the non negative term, we have
\begin{eqnarray*}
&& \hspace{-3.5mm} \!\!\!\!\! c_p\!\int_\Omega|{\mathcal T}f_1-{\mathcal T}f_2|^pdx\leq\!\int_\Omega\Big(|\nabla u_1|^{p-2}\nabla u_1-|\nabla u_2|^{p-2}\nabla u_2,\nabla(u_1-u_2)\Big)dx\\
&=& \!-\int_\Omega (u_1-u_2)\cdot\div\Big(|\nabla u_1|^{p-2}\nabla u_1-|\nabla u_2|^{p-2}\nabla u_2\Big)\, dx\\
&=& \!\!\!\!\!\int_\Omega\! \left[-\div(|\nabla u_1|^{p-2}\nabla u_1)\cdot (u_1-u_2)\!+\!\div(|\nabla u_2|^{p-2}\nabla u_2)\cdot (u_1-u_2) \right] dx\\
&=& \!\!\int_\Omega \left[(f_1-V(x,u_1))\cdot (u_1-u_2)+(-f_2+V(x,u_2))\cdot (u_1-u_2)\right] dx\\
&=& \!\int_\Omega \left[(f_1-f_2)(u_1-u_2)+ (V(x,u_2)-V(x,u_1))\cdot(u_1-u_2)\right]\,dx\\
&\leq & \int_\Omega (f_1-f_2)(u_1-u_2)\, dx,
\end{eqnarray*}
where in the last inequality we have used the monotonicity of $V$ with respect to the second variable, which implies $(V(x,u_2)-V(x,u_1))\cdot(u_1-u_2)\leq 0$.

By Poincar\'e--Sobolev inequality we have
$$
\int_\Omega (f_1-f_2)(u_1-u_2)\, dx\leq c_{1p}||f_1-f_2||_{L^{(p^*)'}}||{\mathcal T}f_1-{\mathcal T}f_2||_{L^p}
$$
so that
\begin{equation*}
||{\mathcal T}f_1-{\mathcal T}f_2||_{L^p(\Omega)}\leq c_{2p}||f_1-f_2||^{\frac1{p-1}}_{L^{(p^*)'}}.
\end{equation*}

\end{proof}

\begin{theorem}\label{thgradientLorentz}
Assume (H1) and (H2) and let $u$ the unique entropic-renormalized solution of the Dirichlet problem \eqref{Eq-datumL1}.

\vspace{1mm}

Let $2\leq p<n$, \  $r\in [1,+\infty]$, \ $1\leq k\le (p^*)'$ \ and \ $\theta=p^*\Big(1-\dfrac1k\Big)$.

\vspace{1mm}

\begin{enumerate}

\item If $0<\theta<1$, i.e.  $1< k< (p^*)'$,
then
$$
||{\mathcal T}f_1-{\mathcal T}f_2||_{L^{k^*(p-1),r(p-1)}}\leq c||f_1-f_2||_{L^{k,r}}^{\frac1{p-1}},
$$
for $f_1,\ f_2$ in $L^{k,r}(\Omega)$, with $k^*=\dfrac{nk}{n-k}$. %if $k<n$ \ and \ $k^*$ any finite number if $k\geq n.$

\vspace{1mm}
\noindent In particular, if $f\in L^{k,r}(\Omega)$, then $ \nabla u\in [L^{k^*(p-1),r(p-1)}(\Omega)]^n$.
% where $u$ is the entropic-renormalized solution $u$ of \eqref{Eq-datumL1}.

\vspace{3mm}
\item If $\theta=0$, i.e. $k=1$, then
$$
{\mathcal T}: G\Gamma(1,r;t^{-1}) \to G\Gamma(\infty,r(p-1);t^{-1},t^{\frac{1}{n'(p-1)}})
$$
is $\frac{1}{p-1}$-H\"olderian.

\vspace{3mm}
\item If $\theta=1$, i.e. $k=(p^*)'$, then
$$
{\mathcal T}: G\Gamma((p^*)',r; v,w)\to G\Gamma(p,r(p-1);v_1,w_1)
$$
is $\frac{1}{p-1}$-H\"olderian, where
$$v=t^{-1}(1-\log t)^\gamma, \ w=(1-\log t)^\beta, \ \gamma>-1, \ \beta\in\mathbb R, \ \gamma +\beta \dfrac{r}{(p^*)'}+1<0,$$
%\vspace{1mm}
%$\lambda=\dfrac{\gamma}{r}+\dfrac{\beta}{(p^*)'}$, \ \
%$\gamma>-1$, $\beta\in\mathbb R$, \ $\gamma +\beta \dfrac{r}{(p^*)'}+1<0$, % so that $\lambda<-1/r$,
$$
\!\!\!v_1=t^{-1}(1-\log t)^{\gamma_1}\!, w_1=(1-\log t)^{\beta_1}\!, \gamma_1>-1, \beta_1\in\mathbb R, \gamma_1 +\frac{\beta_1 r(p-1)}{p}+1<0$$
and such that
\begin{equation}\label{condition exponents}
\frac{\gamma}{r}+\frac{\beta}{(p^*)'}=\frac{\gamma_1}{r}+\frac{\beta_1(p-1)}{p}.
\end{equation}

\end{enumerate}
%\vspace{1mm}
%$\lambda=\dfrac{\gamma_1}{r}+\dfrac{\beta_(p-1)}{(p^*)'}$, \ \ $\gamma_1>-1$, \ $\beta_1\in\mathbb R$, \ $\gamma_1 +\beta_1 \dfrac{r}{(p^*)'}+1<0$, so that $\lambda<-1/r$,

%$$
%{\mathcal T}: G\Gamma((p^*)',r;t^{-1}(1-\log t)^\gamma,(1-\log t)^\beta) \to G\Gamma(p,r(p-1);t^{-1}(1-\log t)^{\gamma_1},(1-\log t)^{\beta_1}).
%$$

\end{theorem}

\vspace{1mm}

\begin{proof}

$$
{\mathcal T}: L^1(\Omega) \to [L^{n'(p-1),\infty}(\Omega)]^n \  \ \hbox{is}  \ \ \frac{1}{p-1}\hbox{-H\"olderian} \ \ \hbox{(by Corollary \ref{cor_extensionT})}
$$

$$
{\mathcal T}:L^{(p^*)'}(\Omega)\to [L^p(\Omega)]^n \  \ \hbox{is}  \ \ \frac{1}{p-1}\hbox{-H\"olderian} \ \ \hbox{(by Theorem \ref{ThTholerianp*'p})}.
$$

%${\mathcal T}: L^1(\Omega) \to [^{n'(p-1),\infty}(\Omega)]^n$  \ is  \ $\frac{1}{p-1}$-H\"olderian
%
%\vspace{2mm}
%
%${\mathcal T}:L^{q'}(\Omega)\to L^p(\Omega)^n$   \ is  \ $\frac{1}{p-1}$-H\"olderian

\vspace{3mm}

It is known from \eqref{identificationLorentz} that, for $0< \theta<1 $,
$$L^{k,r}(\Omega)=(L^1,L^{(p^*)'})_{\theta,r} \, , \ \ \ \ \theta=p^*\Big(1-\dfrac1k\Big)$$
since $\frac{1}{k}=1-\theta +\frac{\theta}{(p^*)'} \ \ \Rightarrow \ \ \frac{1}{k}=1-\theta +\theta\left(1-\frac{1}{p^*}\right) \ \ \Rightarrow \ \ 1-\frac{1}{k}= \frac{\theta}{p^*}$

\vspace{1mm}

\noindent and $0<\theta<1 \ \ \Rightarrow \ \ 1<k<(p*)'  $.

\vspace{2mm}

Moreover, again from \eqref{identificationLorentz}, we have
$$
L^{k^*(p-1),r(p-1)}=\left(L^{n'(p-1),\infty},L^p\right)_{\theta,r(p-1)},
$$
in fact $\frac{1}{k^*(p-1)}=\frac{1-\theta}{n'(p-1)}+\frac{\theta}{p}$, with $k^*=\dfrac{nk}{n-k}$ since $k<(p^*)'=\frac{np}{np-n+p}<n$.

It is easy to see that
\begin{eqnarray*}
 \frac{1-\theta}{n'(p-1)}+\frac{\theta}{p} &=& \frac{n-1}{n(p-1)}+\theta\left(\frac{1}{p}-\frac{n-1}{n(p-1)}\right)\\
&=& \frac{n-1}{n(p-1)}+ \theta\left(\frac{np-n-np+p}{np(p-1)}\right)\\
&=&  \frac{n-1}{n(p-1)}-  \frac{\theta(n-p)}{np(p-1)}= \frac{np-p-\left(1-\frac{1}{k}\right)\frac{np}{n-p}(n-p)}{np(p-1)}\\
&=& \frac{n-k}{nk}\cdot \frac{1}{p-1}=\frac{1}{k^*(p-1)}.
\end{eqnarray*}
Then, for $f_1,\ f_2$ in $L^{k,r}(\Omega)$, from Theorem \ref{AFFGHR_Th2-2} with $\lambda=0$ and $\alpha=\frac{1}{p-1}$, $L^{(p^*)'}\subset L^1$, \ $L^p\subset L^{n'(p-1),\infty}$, we have
$$
{\mathcal T}: (L^1,L^{(p^*)'})_{\theta,r} \to \left[ \left(L^{n'(p-1),\infty},L^p\right)_{\theta,r(p-1)}\right]^n
$$
is also $\frac{1}{p-1}$-H\"olderian, which implies
$$
{\mathcal T} : L^{k,r}(\Omega) \to [L^{k^*(p-1),r(p-1)}(\Omega)]^n
$$
 and
$$||{\mathcal T}f_1-{\mathcal T}f_2||_{\theta,r(p-1)}\lesssim ||f_1-f_2||^{\frac1{p-1}}_{\theta,r},$$
which yields
$$||{\mathcal T}f_1-{\mathcal T}f_2||_{L^{k^*(p-1),r(p-1)}(\Omega)}\lesssim ||f_1-f_2||^{\frac1{p-1}}_{L^{k,r}(\Omega)},$$
since the bounded functions are dense in $L^{k,r}(\Omega)$. %, we complete the proof.

\vspace{2mm}

For $\theta =0$ we have from Theorem \ref{AFFGHR_Th2-2}, with $\lambda=0$ and $\alpha=\frac{1}{p-1}$,
$$
{\mathcal T}: (L^1,L^{(p^*)'})_{0,r} \to \left[ \left(L^{n'(p-1),\infty},L^p\right)_{0,r(p-1)}\right]^n
$$
and the assertion follows by \eqref{L1LpGGamma} and \eqref{LorentzLebesguetheta0} since
$$(L^1,L^{(p^*)'})_{0,r}=G\Gamma(1,r;t^{-1}),$$ $$\left(L^{n'(p-1),\infty},L^p\right)_{0,r(p-1)}=G\Gamma(\infty,r(p-1);t^{-1},t^{\frac{1}{n'(p-1)}}).$$

\vspace{2mm}

For $\theta =1$ we have from Theorem \ref{AFFGHR_Th2-2}, with $\lambda<-\frac{1}{r}$ and  $\alpha=\frac{1}{p-1}$,
$$
{\mathcal T}: (L^1,L^{(p^*)'})_{1,r;\lambda} \to \left[ \left(L^{n'(p-1),\infty},L^p\right)_{1,r(p-1);\frac{\lambda}{p-1}}\right]^n
$$
and the assertion follows by \eqref{Ggammatheta1} since
$$(L^1,L^{(p^*)'})_{1,r;\lambda}=G\Gamma((p^*)',r;t^{-1}(1-\log t)^\gamma,(1-\log t)^\beta),$$
$$\left(L^{n'(p-1),\infty},L^p\right)_{1,r(p-1);\frac{\lambda}{p-1}}=G\Gamma(p,r(p-1);t^{-1}(1-\log t)^{\gamma_1},(1-\log t)^{\beta_1}),$$
where

$\lambda=\dfrac{\gamma}{r}+\dfrac{\beta}{(p^*)'}$, \ \ $\gamma>-1$, $\beta\in\mathbb R$, \ $\gamma +\beta \dfrac{r}{(p^*)'}+1<0$, so that $\lambda<-1/r$,

$\frac {\lambda}{p-1}=\dfrac{\gamma_1}{r(p-1)}+\dfrac{\beta_1}{p}$, \ \ $\gamma_1>-1$, $\beta_1\in\mathbb R$, \ $\gamma_1 +\beta_1 \dfrac{r(p-1)}{p}+1<0$, so that $\frac {\lambda}{p-1}<-\frac{1}{r(p-1)}$, i.e. $\lambda<-1/r$,

with the condition \eqref{condition exponents}:   %, which easily follows by the equality
$$
\frac{\gamma}{r}+\frac{\beta}{(p^*)'}=\frac{\gamma_1}{r}+\frac{\beta_1(p-1)}{p}, \ \ \ \ \hbox{with} \ \  (p^*)'=\frac{np}{np-n+p}.
$$
%which is also equivalent to
%\begin{equation*}
%\frac{\gamma-\gamma_1}{r}+\frac{(\beta-\beta_1)(p-1)}{p}+\frac{\beta}{n}=0.
%\end{equation*}

\end{proof}

\vspace{1mm}

\begin{remark}
 In the case $k=r\in]1,(p^*)'[$ we improve previous known results, in fact the usual estimate was only obtained in $[L^{r^*(p-1)}(\Omega)]^n$ (see \cite{Cianchi-Mazja_Arch}) and
$L^{r^*(p-1),r(p-1)}(\Omega)\subset L^{r^*(p-1)}(\Omega)$.
\end{remark}

\vspace{1mm}

\begin{remark}
If $p=\frac{2n}{n+1}$, we have $p<n$ and $(p^*)'= \frac{np}{np-n+p}=p$. Therefore, in this particular case, the condition \eqref{condition exponents} is certainly satisfied if $\gamma=\gamma_1$ and $\beta=\beta_1(p-1)$.
\end{remark}

\vspace{2mm}

The identification of interpolation spaces between couples of Lebesgue or Lorentz spaces,
recovering spaces such as Lorentz–Zygmund spaces or $G\Gamma$ spaces, permit us to
obtain precise regularity of the gradient of an entropic-renormalized solution.

\begin{theorem}\label{Th_regularity_gradient}

Let  $2\leq p<n$. Assume (H1) and (H2) and let $u$ the unique entropic-renormalized solution of the Dirichlet problem \eqref{Eq-datumL1}.
%$$
%-\Delta_p u +V(x;u)=f\in L^1(\Omega), \quad\quad  u=0 \quad {\rm on} \ \partial\Omega.
%$$

Let ${\mathcal T}$ the mapping $f\mapsto {\mathcal T} f$, with
$${\mathcal T}f=\nabla u.$$

\begin{description}[before={\renewcommand\makelabel[1]{\bfseries ##1}}]

\item [(1)] \ Let $0<\theta<1$, \ $1\leq q<\infty, \ \lambda\in\mathbb R$, \ $\displaystyle\frac{1}{p_\theta}=\frac{(1-\theta)}{n'(p-1)}+\frac{\theta} {p} $, \ $n'=\displaystyle \frac{n}{n-1}$.
\end{description}
%\vspace{2mm}

 % Let $ m=(p^*)'=\frac{np}{(n+1)p-n}$ \ if \ $2\leq p<n$ \ and \ $m\in [1,+\infty[$
%\ if \ $p\geq n$.

\vspace{2mm}
Then
$$
\displaystyle{\mathcal T}: L^{\frac{p^*}{p^*-\theta},q}(\log L)^\lambda \to \left[L^{p_\theta,q(p-1)}(\log L)^{\frac\lambda{p-1}}\right]^n
$$
is $\frac {1}{p-1}$-H\"olderian, where $p^*=\frac{np}{n-p}$.

\vspace{5mm}

\begin{description}[before={\renewcommand\makelabel[1]{\bfseries ##1}}]
\item [(2)] \ Let $\theta=0$,  \ $\lambda\geq -\frac{1}{p}$. Then
\end{description}
$$
{\mathcal T}: G\Gamma(1,p;t^{-1}(1-\log t)^{\lambda p})\to (L^{n'(p-1),\infty}(\Omega), L^p(\Omega))_{0,p(p-1);\frac{\lambda}{p-1}}
$$

\vspace{2mm}
is  $\frac{1}{p-1}-$H\"olderian. Moreover, observing that
$$
(L^{n'(p-1),\infty};L^p)_{0,p(p-1);\frac{\lambda}{p-1}}=G\Gamma(\infty,p(p-1);t^{-1}(1-\log t)^{\lambda p},t^{\frac{1}{n'(p-1)}}),
$$
an equivalent norm of $\nabla u$, for $\sigma$ such that $\frac{1}{\sigma}=\frac{p'}{pn'}-\frac{1}{p}$, is given by
\begin{eqnarray*}
& & \|\nabla u\|_{(L^{n'(p-1),\infty}(\Omega), L^p(\Omega))_{0,p(p-1);\frac{\lambda}{p-1}}} \\
& & \approx  \left[ \int_0^1 \left( \sup_{0<s<t^\sigma} s^{\frac{p'}{pn'}}|\nabla u|_*(s)(1-\log t)^{\frac{\lambda}{p-1}}\right)^{p(p-1)}\frac{dt}{t}\right]^{\frac{1}{p(p-1)}}.
\end{eqnarray*}

\vspace{5mm}

\begin{description}[before={\renewcommand\makelabel[1]{\bfseries ##1}}]
\item [(3)] \ Let $\theta=1$, \ $\lambda < -\frac{1}{p}$. Then
 $$
{\mathcal T}: (L^{1}(\Omega), L^{(p^*)'}(\Omega))_{1,p;\lambda}\to (L^{n'(p-1),\infty}(\Omega), L^p(\Omega))_{1,p(p-1);\frac{\lambda}{p-1}}
$$
 is $\frac{1}{p-1}-$H\"olderian and we have
 \begin{eqnarray*}
& & \left[ \int_0^1\left(\left(\int_t^1 |\nabla u|_\ast(s)^{p}ds\right)^{\frac{1}{p}}
(1-\log t)^{\frac{\lambda}{p-1}}\right)^{p(p-1)} \frac{dt}{t}\right]^{\frac{1}{p(p-1)}}\\
& \leq & c \left[ \int_0^1\left(\left(\int_t^1 f_\ast(s)^{(p^*)'}ds\right)^{\frac{1}{(p^*)'}}(1-\log t)^{\lambda }\right)^p \frac{dt}{t}\right]^{\frac{1}{p}}
\end{eqnarray*}

\end{description}

\end{theorem}

\vspace{3mm}

\begin{proof}

Let $0<\theta<1$.
$$ {\mathcal T}: L^1(\Omega)\longrightarrow [L^{n'(p-1),\infty}(\Omega)]^n \ \ \hbox{is} \ \ \frac{1}{p-1}-\hbox{H\"olderian} \ \ \ \hbox{(by Corollary \ref{cor_extensionT})}
$$
$$
{\mathcal T}:L^{(p^*)'}(\Omega)\to [L^p(\Omega)]^n \  \ \hbox{is}  \ \ \frac{1}{p-1}\hbox{-H\"olderian} \ \ \hbox{(by Theorem \ref{ThTholerianp*'p})}
$$

\vspace{2mm}

\noindent The smooth functions are dense in the Lorentz-Zygmund spaces $L^{p,q}(\log L)^\lambda$\!,  $1\leq p,q<\infty$. Then
$$
{\mathcal T}: (L^1(\Omega),L^{(p^*)'}(\Omega))_{\theta,q;\lambda} \to (L^{n'(p-1),\infty}(\Omega),L^{p}(\Omega))_{\theta,\frac{q}{\alpha};\lambda \alpha}
$$
is $\alpha=\dfrac{1}{p-1}$- H\"olderian.

 \vspace{3mm}

 Moreover, we identify
$$
(L^1(\Omega),L^{(p^*)'}(\Omega))_{\theta,q;\lambda}=L^{\frac{p^*}{p^*-\theta},q}(\log L)^\lambda=L^{\frac{np}{np-\theta(n-p)},q}(\log L)^\lambda \ \ \ \ \hbox{(by \eqref{identification_L_L_LZ_3_parameters})},
$$

$$
(L^{n'(p-1),\infty};L^p)_{\theta,\frac{q}{\alpha};\lambda \alpha}=L^{p_\theta,\frac{q}{\alpha}}(\log L)^{\lambda\alpha}, \ \ \ \frac1{p_\theta}=\frac{1-\theta}{n'(p-1)}+\frac\theta p \ \ \ \ \hbox{(by \eqref{identification_LrinftyLm})}.
$$
Then, by Theorem \ref{AFFGHR_Th2-2} with $\alpha=\frac{1}{p-1}$, we get
$$
{\mathcal T}: L^{\frac{p^*}{p^*-\theta},q}(\log L)^\lambda \to \left[L^{p_\theta,q(p-1)}(\log L)^{\frac\lambda{p-1}}\right]^n \ \ \hbox{is} \ \ \frac{1}{p-1}\hbox{-H\"olderian}.
$$

\vspace{1mm}

In the case $\theta=0$,  \ $\lambda\geq -\frac{1}{p}$, by \eqref{L1LpGGamma}, we have
$$
(L^1(\Omega),L^{(p^*)'}(\Omega))_{0,p;\lambda}=G\Gamma(1,p;t^{-1}(1-\log t)^{\lambda p})
$$
and $L^{(p^*)'}$ is dense therein. By \eqref{LorentzLebesguetheta0},
$$
(L^{n'(p-1),\infty};L^p)_{0,p(p-1);\frac{\lambda}{p-1}}=G\Gamma(\infty,p(p-1);t^{-1}(1-\log t)^{\lambda p},t^{\frac{1}{n'(p-1)}}).
$$

\vspace{1mm}
Therefore, by Theorem \ref{AFFGHR_Th2-2}, % \  ${\mathcal T}: (X_0,X_1)_{\theta,p;\lambda} \to   (Y_0,Y_1)_{\theta\frac\alpha\beta,\frac p\alpha;\lambda\alpha}$,
with $\alpha=\frac{1}{p-1}$, the mapping
$$
{\mathcal T}: (L^1(\Omega),L^{(p^*)'}(\Omega))_{0,p;\lambda} \to   (L^{n'(p-1),\infty};L^p)_{0,p(p-1);\frac{\lambda}{p-1}}
$$
is $\dfrac{1}{p-1}$- H\"olderian. By the identification of the above interpolation spaces, we have
$$
{\mathcal T}: G\Gamma(1,p;t^{-1}(1-\log t)^{\lambda p}) \to  G\Gamma(\infty,p(p-1);t^{-1}(1-\log t)^{\lambda p},t^{\frac{1}{n'(p-1)}})
$$
and the assertion follows.

\vspace{2mm}

The same argument holds for $\theta=1$.
\end{proof}

\begin{remark}
We observe that, in the case $0<\theta<1$, for $\lambda=0$, we recover the conclusion of Theorem \ref{thgradientLorentz}, in fact $k=\frac{p^*}{p^*-\theta}$ and $\frac{1}{k^*(p-1)}=\frac{1}{p_\theta}$, so in this case, for $q=r$, we have
$$
L^{\frac{p^*}{p^*-\theta},q}(\log L)^\lambda= L^{k,r}, \ \ \ L^{p_\theta,q(p-1)}(\log L)^{\frac\lambda{p-1}}=L^{k^*(p-1),r(p-1)},
$$
 while, for $\lambda\neq 0$, we have
$$L^{\frac{p^*}{p^*-\theta},q}(\log L)^\lambda\subset L^{k,r}, \ \ \ L^{p_\theta,q(p-1)}(\log L)^{\frac\lambda{p-1}}\subset L^{k^*(p-1),r(p-1)} \ \ \hbox{if} \ \lambda>0,  $$
$$ L^{k,r}\subset L^{\frac{p^*}{p^*-\theta},q}(\log L)^\lambda , \ \ \ L^{k^*(p-1),r(p-1)}\subset L^{p_\theta,q(p-1)}(\log L)^{\frac\lambda{p-1}} \ \ \hbox{if} \ \lambda<0
$$
(see inclusion relations in Section \ref{Sec function spaces}).

\end{remark}

\vspace{2mm}

\begin{remark}
We observe that, by \eqref{Ggammatheta1},
$$
(L^1,L^{(p^*)'})_{1,p;\lambda}=G\Gamma((p^*)',p;t^{-1}(1-\log t)^{\gamma_1},(1-\log t)^{\beta_1}),
$$
$$
(L^{n'(p-1),\infty}, L^p)_{1,p(p-1);\frac{\lambda}{p-1}}=G\Gamma(p,p(p-1);t^{-1}(1-\log t)^{\gamma_2},(1-\log t)^{\beta_2})
$$
where

\vspace{1mm}
$\lambda=\frac{\gamma_1}{p}+\frac{\beta_1}{(p^*)'}$,\  $\gamma_1>-1$, \ $\beta_1\in \mathbb R$, \
$\gamma_1+\beta_1\frac{p}{(p^*)'}+1<0$, \ $\lambda<-\frac{1}{p}$,

\vspace{1mm}
$\lambda=\frac{\gamma_2}{p}+\frac{\beta_2(p-1)}{p}$,\  $\gamma_2>-1$,\  $\beta_2\in \mathbb R$, \ $\gamma_2+\beta_2(p-1)+1<0$,\  $\lambda<-\frac{1}{p}$,

$$\gamma_1-\gamma_2+(\beta_1-\beta_2)(p-1)+\beta_1\frac{p}{n}=0.$$
The last condition follows by the equality
$$
\frac{\gamma_1}{p}+\frac{\beta_1}{(p^*)'}=\frac{\gamma_2}{p}+\frac{\beta_2(p-1)}{p}, \ \ \ \ \hbox{with} \ \  (p^*)'=\frac{np}{np-n+p}.
$$

\end{remark}

\vspace{5mm}

To obtain boundedness of the solution in a more general situation, stated in next Theorem, we need to assume the growth condition (H3).

\begin{theorem}\label{Th_boundednessH3}

Let  $2\leq p<n$. Assume (H1), (H2) and (H3) and let $u$ the entropic-renormalized solution of the Dirichlet problem \eqref{Eq-datumL1}.
%\vspace{2mm}
%
%{\color{red}\centerline{$-\Delta_p u +V(x;u)=f\in L^1(\Omega), \quad\quad  u=0 \quad {\rm on} \ \partial\Omega$.}}
%\vspace{2mm}

Let ${\mathcal T}$ the mapping $f\mapsto {\mathcal T} f$, with
$${\mathcal T}f=\nabla u.$$

\vspace{2mm}
Let \ $1<q<+\infty$, $\lambda\in \mathbb R$, \ $0\leq \theta <1$.

\vspace{2mm}

%\setbeamersize{description width=0.2cm}

\begin{itemize}

\item If $\displaystyle f\in L^{\frac{n'}{n'-\theta},q}(\log L)^\lambda $, \ $0<\theta<1$, then

$$  \nabla u \in L^{\frac{n(p-1)}{(1-\theta)(n-1)},q(p-1)}(\log L)^{\frac{\lambda}{p-1}}.$$

\vspace{2mm}

\item If $\displaystyle f\in G\Gamma(1,q;t^{-1}(1-\log t)^{\lambda q}) $, then

$$
\nabla u \in  G\Gamma(\infty,q(p-1);t^{-1}(1-\log t)^{\lambda q},t^{\frac{1}{n'(p-1)}}).
$$

\end{itemize}

\end{theorem}

\vspace{2mm}

\begin{proof}

$$
{\mathcal T}: L^1(\Omega)\longrightarrow [L^{n'(p-1),\infty}(\Omega)]^n \ \ \hbox{is} \ \  \frac{1}{p-1}-\hbox{H\"olderian} \ \  \hbox{(by Corollary \ref{cor_extensionT})},
$$

$$ {\mathcal T}: L^{n,1}(\Omega)\longrightarrow [L^{\infty}(\Omega)]^n, \ \ \  \hbox{is bounded (by \eqref{gradientbounded} in Proposition \ref{prop_u_bounded})}.$$

\vspace{1mm}

Using Theorem \ref{AFFGHR_Th2-2}, that is
$$
{\mathcal T}: (X_0,X_1)_{\theta,q;\lambda}\to (Y_0,Y_1)_{\theta\frac\alpha\beta,\frac {q}\alpha;\lambda\alpha}
$$
is bounded, for $\beta=\alpha=\frac{1}{p-1}$, we have
$$ {\mathcal T}: (L^1(\Omega),L^{n,1}(\Omega))_{\theta,q;\lambda}\longrightarrow (L^{n'(p-1),\infty}(\Omega),L^{\infty}(\Omega))_{\theta,q(p-1);\frac{\lambda}{p-1}} $$
is bounded.

\vspace{4mm}
If $0<\theta<1$, from \eqref{LZinterpolationLorentz} and \eqref{identification_LrinftyLinfty} respectively, we have
$$
(L^1,L^{n,1})_{\theta,q;\lambda}=L^{\frac{n'}{n'-\theta},q}(\log L)^\lambda
$$
%\vspace{2mm}
%{\small\centerline{
%$(L^1,L^{n,1})_{\theta,p_2;\lambda}=
%[(L^{n',\infty},L^\infty)_{1-\theta,p_2^{'};-\lambda}]^{'}=
%[L^{\frac{n'}{\theta},p_2^{'}}(\log L)^{-\lambda}]^{'}=L^{(\frac{n'}{\theta})',p_2}(\log L)^{\lambda}=L^{\frac{n'}{n'-\theta},p_2}(\log L)^\lambda$
%}}
%\vspace{2mm}
$$(L^{n'(p-1),\infty},L^\infty)_{\theta,q(p-1);\frac{\lambda}{p-1}}=
L^{\frac{n'(p-1)}{1-\theta},q(p-1)}(\log\,L)^{\frac{\lambda}{p-1}}.$$

\vspace{3mm}
If $\theta=0$, by \eqref{LorentLorentz_GGamma} and \eqref{LorentzLebesguetheta0} respectively, we have
$$
(L^1(\Omega),L^{n,1}(\Omega))_{0,q;\lambda}=G\Gamma(1,q;t^{-1}(1-\log t)^{\lambda q}),
$$
$$
(L^{n'(p-1),\infty};L^p)_{0,q(p-1);\frac{\lambda}{p-1}}=G\Gamma(\infty,q(p-1);t^{-1}(1-\log t)^{\lambda q},t^{\frac{1}{n'(p-1)}}).
$$

\end{proof}

\begin{remark}
We recall that, as point out at the beginning of Section \ref{SecApplHolderianpgreater2}, all the previous results are also true for $p\geq n$, and in this case the symbol $p^*$ must be considered as any finite number. Hence, if in (1) of Theorem \ref{Th_regularity_gradient} we consider $p\geq n$ and chose $p^*=n'$, comparing the result with the first one in Theorem \ref{Th_boundednessH3} for $p\geq n$, we have the datum $f$ in the same space $L^{\frac{n'}{n'-\theta},q}(\log L)^\lambda$, while the gradient of the solution $u$ in Theorem \ref{Th_boundednessH3} belongs to a smaller space, since
$$
L^{\frac{n(p-1)}{(1-\theta)(n-1)},q(p-1)}(\log L)^{\frac{\lambda}{p-1}}\subset L^{p_\theta,q(p-1)}(\log L)^{\frac\lambda{p-1}}
$$
with $\frac{1}{p_\theta}=\frac{(1-\theta)}{n'(p-1)}+\frac{\theta} {p}$.

\vspace{1mm}

Moreover, $\frac{n(p-1)}{(1-\theta)(n-1)}\geq\frac{n}{1-\theta}>n$ \ for \ $p\geq n$ and $0<\theta<1$. Therefore $\nabla u\in L^{r}$, \ $r>n$, and by Sobolev theorem, it follows that the solution $u$ is bounded.

\end{remark}

\vspace{2mm}

\section{The H\"olderian mappings for the case \texorpdfstring{$1<p< 2$}{R\texttwosuperior}}\label{SecApplHolderianpsmaller2}

Some of results for $2\leq p<n$ remain true in the case $1 < p < 2$. The fundamental changes concern the H\"older properties than can exist but are not sharp as for the case $p \geq 2$, and the H\"older constant appearing depends on the data.

\begin{theorem} {\rm (local Lipschitz contraction)}\label{ThlocalLipschitz}

Let $1<p<2$, \ $n\geq 2, \  p^*=\frac{np}{n-p}, \ \ (p^*)^{'}=\frac{np}{np+p-n}$.

\vspace{2mm}

 Let $f_1,f_2\in L^{(p^*)^{'}}(\Omega)$ and $V$ satisfies (H1) and (H2).

\vspace{2mm}

Let $u_i, \ i=1,2$, be the \textbf{weak solution} of
$$-\Delta_p u_i+V(x;u_i)=f_i, \quad\quad  u_i=0 \quad {\rm on} \ \partial\Omega . $$

Then
\vspace{1mm}

\begin{itemize}

\item $||\nabla {u_1}||_{L^p}\leq c||f_1||^{\frac{1}{p-1}}_{L^{(p^*)^{'}}}, \ \ \ \  ||\nabla {u_2}||_{L^p}\leq c||f_2||^{\frac{1}{p-1}}_{L^{(p^*)^{'}}}$

\vspace{2mm}

\item $||\nabla (u_1-u_2)||_{L^p}\leq c\Big(||\nabla u_1||_{L^p}+||\nabla u_2||_{L^p}\Big)^{2-p}||f_1-f_2||_{L^{(p^*)^{'}}}$.

\end{itemize}

\vspace{2mm}

Here the constant $c$ depends only on $p$ and $\Omega$.
\end{theorem}

\vspace{2mm}

\begin{corollary} {\rm (of Theorem \ref{ThlocalLipschitz})}\label{CorlocalLipschitz}

Under the same assumptions as in Theorem \ref{ThlocalLipschitz}, there exists a constant $c$ depending only on $p$ and $\Omega$ such that
\begin{equation}\label{ineqlocalLipschitz}
||\nabla (u_1-u_2)||_{L^p}\leq c\Big(||f_1||^{\frac1{p-1}}_{L^{(p^*)'}}+
||f_2||^{\frac1{p-1}}_{L^{(p^*)'}}\Big)^{2-p}||f_1-f_2||_{L^{(p^*)^{'}}}.
\end{equation}

\vspace{1mm}

In particular,
$${\mathcal T}: L^{(p^*)^{'}}(\Omega)\to [L^p(\Omega)]^n \ \ \  \hbox{is locally Lipschitz}.$$

\end{corollary}

\vspace{2mm}

\begin{theorem}
Let $1<p<2$,\ \ $r\in[1,+\infty]$ and $k$ such that $(p^*)^{'}<k<n$. Assume (H1), (H2) and (H3).

Let $u$ be the \textbf{entropic-renormalized solution} of the equation
$$-\Delta_p u +V(x;u)=f, \quad\quad  u=0 \quad {\rm on} \ \partial\Omega.$$
Then, the non linear mapping ${\mathcal T}$ is bounded from $L^{k,r}(\Omega)$ into $L^{k_1,r}(\Omega)$,

\vspace{2mm}

with \ $k_1=\dfrac{p}{1-\theta(p-1)}$, \ \  \ $\theta=\dfrac{\frac{1}{(p^*)'}-\frac{1}{k}}{\frac{1}{(p^*)'}-\frac{1}{n}}\in ]0,1[. $ % \ $\theta$ defined by $\dfrac{1}{k}=\dfrac{1-\theta}{(p^*)'}+\dfrac{\theta}{n}$

\end{theorem}

\vspace{2mm}

\begin{proof}
From Corollary \ref{CorlocalLipschitz} we have
$${\mathcal T}: L^{(p^*)^{'}}(\Omega)\to [L^p(\Omega)]^n \ \ \  \hbox{is locally Lipschitz}$$
and \eqref{ineqlocalLipschitz} holds. By Proposition \ref{prop_u_bounded} we have
$${\mathcal T}:L^{n,1}(\Omega)\to[L^\infty(\Omega)]^n \ \ \ \hbox{is bounded}. $$
By Theorem \ref{AFFGHR_Th2-1} with $\lambda=0$, $\alpha=1$, $\beta=\frac{1}{p-1}$ we have that
$${\mathcal T} \ \ \hbox{maps} \ \  (X_0,X_1)_{\theta,r;\lambda} \ \ \hbox{into} \ \ (Y_0,Y_1)_{\theta\frac\alpha\beta,\frac r\alpha;\lambda\alpha},$$
i.e.
$${\mathcal T} : (L^{(p^*)^{'}}(\Omega),L^{n,1}(\Omega))_{\theta,r} \to (L^p(\Omega),L^\infty(\Omega))_{\theta(p-1),r}$$

\noindent is a locally bounded mapping, with $\theta\in]0,1[$. The assertion follows from \eqref{identificationLorentz}, since
$$
(L^{(p^*)^{'}}(\Omega),L^{n,1}(\Omega))_{\theta,r} =L^{k,r}(\Omega), \ \ \ \ \frac{1}{k}=\frac{1-\theta}{(p^*)'}+\frac{\theta}{n}
$$
and
$$
(L^p(\Omega),L^\infty(\Omega))_{\theta(p-1),r}=L^{k_1,r}(\Omega), \ \ \ \ \frac{1}{k_1}=\frac{1-\theta(p-1)}{p}.
$$
\end{proof}

\vspace{2mm}

For other results concerning equations with data in Lorentz spaces,
see, e.g., \cite{Ferone_Murat2014,Grenon_Murat_Porretta}.

\vspace{2mm}

\section{Conclusion}

We conclude by highlighting that in \cite{AFFGER2023} also applications to the anisotropic equation and variable exponents version of the $p$-Laplacian are given. We refer the reader to the paper \cite{AFFGER2023} for the results.

\vspace{0.1cm}

 \textbf{Anisotropic equation}

\begin{equation*}\label{anisotropic equation}
\begin{cases}-\Delta_\vecp u+V(x;u)=f & \ {\rm in}\ \Omega\\
u=0 & \  {\rm on} \ \partial\Omega,
\end{cases}
\end{equation*}
where
$$ \Delta_\vecp u=-\sum_{i=1}^n\frac \partial{\partial x_i}\left(\left|\frac{\partial  u}{\partial x_i}\right|^{p_i-2}\frac{\partial u}{\partial x_i}\right),$$

\noindent $\vecp=(p_1,\ldots,p_n)$, \  $1<p_i<+\infty$, \   $\vec {p'}=(p^{'}_1,\ldots,p^{'}_n)$, \ $p^{'}_i$ is the conjugate of $p_i$.

\vspace{0.3cm}
\textbf{Variable exponents version the $p$-Laplacian}
\begin{equation*}
 \begin{cases}-\Delta_{p(\cdot)} u+V(x;u)=f & \ {\rm in}\ \Omega\\
 u=0 & \  {\rm on} \ \partial\Omega,
\end{cases}
\end{equation*}
where
$\Delta_{p(\cdot)} u=\div(|\nabla u|^{p(x)-2}\nabla u)$.

\bibliographystyle{plainurl}
 %\nocite{*} % dice a bibtex di estrarre tutti gli elementi del database bibliografico, anche se non sono citati nel documento

\bibliography{FormicaAGMA2024}

\end{document}